\documentclass[a4paper,12pt]{article}

\usepackage[text={6.5in,9.5in},centering]{geometry}
\usepackage{graphicx,url,subfig,mathtools}
\RequirePackage[T1]{fontenc}
\RequirePackage{amsthm,amsmath,amssymb,amscd}
\RequirePackage[numbers]{natbib}
\RequirePackage[colorlinks,citecolor=blue,urlcolor=blue]{hyperref}
\usepackage{enumerate}
\usepackage{datetime}

\usepackage[normalem]{ulem} 
\usepackage{soul} 
\makeatother
\numberwithin{equation}{section}
\allowdisplaybreaks

\theoremstyle{plain}
\newtheorem{theorem}{Theorem}[section]

\newtheorem{lemma}[theorem]{Lemma}
\newtheorem{proposition}[theorem]{Proposition}

\theoremstyle{definition}

\newtheorem{remark}[theorem]{Remark}

\newcommand{\R}{\mathbb{R}}

\def\cD{\mathcal{D}}
\def\cE{\mathcal{E}}
\def\cF{\mathcal{F}}
\def\cG{\mathcal{G}}

\def\cH{\mathcal{H}}
\def\cP{\mathcal{P}}

\def\bE{\mathbb{E}}

\def\bR{\mathbb{R}}

\def\bE{\mathbb E}

\def\bP{\mathbb P}
\def\bR{\mathbb R}

\def\cA{\mathcal A}
\def\cB{\mathcal B}
\def\cD{\mathcal D}
\def\cF{\mathcal F}
\def\cH{\mathcal H}

\newcommand{\les}{\lesssim}

\def\x{\textbf x}

\def\z{\textbf z}

\title{Gaussian fluctuations for the wave equation \\
under rough random perturbations}

\author{Raluca M. Balan\footnote{Corresponding author. University of Ottawa, Department of Mathematics and Statistics,
Ottawa, ON, K1N 6N5, Canada. E-mail address: rbalan@uottawa.ca.} \footnote{Research
supported by a grant from the Natural Sciences and Engineering
Research Council of Canada.}
\and
Jingyu Huang\footnote{University of Birmingham,  School of Mathematics, Birmingham, B15 2TT, United Kingdom. \break E-mail address: j.huang.4@bham.ac.uk.}
 \and
Xiong Wang\footnote{Johns Hopkins University, Department of Mathematics, Baltimore, MD, 21218, United States. \break E-mail address: xiong\_wang@jhu.edu.}
\and
Panqiu Xia\footnote{Auburn University, Department of Mathematics and Statistics, Auburn, AL, 36849, United States. \break E-mail address: pqxia@auburn.edu.}
\and
Wangjun Yuan\footnote{University of Luxembourg, Department of Mathematics,
L-4364 Esch-sur-Alzette, Luxembourg. \break E-mail address: ywangjun@connect.hku.hk.}
}

\begin{document}
\maketitle

\begin{abstract}
\noindent In this article, we consider the stochastic wave equation in spatial dimension $d=1$, with linear term $\sigma(u)=u$ multiplying the noise. This equation is driven by a Gaussian noise which is white in time and fractional in space with Hurst index $H \in (\frac{1}{4},\frac{1}{2})$. First, we prove that the solution is strictly stationary and ergodic in the spatial variable. Then, we show that with proper normalization and centering, the spatial average of the solution converges to the standard normal distribution, and we estimate the rate of this convergence in the total variation distance. We also prove the corresponding functional convergence result.
\end{abstract}



\maketitle

\tableofcontents

\section{Introduction}

The study of stochastic partial differential equations (SPDEs) is an active research area in stochastic analysis which has been growing steadily in the last 40 years, using two approaches: the semigroup approach developed by Da Prato and Zabczyk in \cite{DZ92}, and the random field approach initiated in Walsh's lecture notes \cite{Walsh86}. The two approaches rely on different infinite-dimensional extensions of It\^o's martingale theory of stochastic integration, and lead to different concepts of solutions. A comparison between the two approaches can be found in \cite{dalang-quer11}. Classical equations which have been studied using the random field approach are: the stochastic heat equation (SHE), and the stochastic wave equation (SWE). When these equations are perturbed by a space-time Gaussian white noise, random field solutions exist only in spatial dimension $d=1$. A systematic study of SPDEs in higher dimensions was initiated by Dalang in the seminal article \cite{Dalang99}, by considering a spatially-homogeneous Gaussian noise with spatial covariance given by a  non-negative-definite function $\gamma:\bR^d \to [0,\infty]$. A typical example is the Riesz kernel $\gamma(x)=|x|^{-\beta}$ with $\beta \in (0,d)$. The case $\gamma=\delta_0$, where $\delta_0$ is the Dirac distribution at $0$, corresponds formally to the white noise in space. Subsequent investigations revealed that the solutions to these equations have many interesting properties, such as: intermittency \cite{FK09}, H\"older continuity \cite{dalang-sanz09}, strict positivity \cite{chen-kim}, dense blow-up \cite{CHKK}, to mention just a few. Recent investigations focus on relaxing the conditions on the coefficients of the equation, as for instance in \cite{DKZ,FKN,salins21}.

\medskip

In the 1990's, {\em fractional Brownian motion} (fBm) became a popular model for the noise in various problems in stochastic analysis. Recall that a fBm is a zero-mean Gaussian process $(B_x^{(H)})_{x\in \bR}$ with covariance
\[
\bE[B_x^{(H)} B_y^{(H)}]=\frac{1}{2}(|x|^{2H}+|y|^{2H}-|x-y|^{2H})
\eqqcolon R_H(x,y),
\]
where $H \in (0,1)$ is the Hurst index. If $H=1/2$, the fBm is a Brownian motion. The paths of fBm are H\"older continuous of order less than $H$, and hence, are smoother or rougher than the Brownian paths, depending on whether $H>1/2$ or $H<1/2$. In the ``regular'' case $H>1/2$, the covariance of fBm can be written as:
\begin{equation}
\label{cov-fbm}
R_H(x,y)=H(2H-1)\int_{0}^{x} \int_{0}^{y} |u-v|^{2H-2}du dv,
\end{equation}
using the Riesz kernel $\gamma(x)=H(2H-1)|x|^{2H-2}$, which is the second derivative $(|x|^{2H})''$ in the sense of distributions. In the ``rough'' case $H<1/2$, \eqref{cov-fbm} does not hold since $(|x|^{2H})''$ is not a function. In this case, it is useful to work with the spectral representation:
\begin{equation}
\label{spectral-cov}
R_H(x,y)=c_H \int_{\bR} \cF 1_{[0,x]}(\xi) \overline{\cF 1_{[0,y]}(\xi)}|\xi|^{1-2H}d\xi,
\end{equation}
where $c_H = \frac{\Gamma(2H+1)\sin(\pi H)}{2\pi}$ and $\cF \varphi (\xi)=\int_{\bR}e^{-i \xi x} \varphi(x)dx$ is the Fourier transform.

The fBm is not a semi-martingale, and It\^o calculus cannot be used. Two methods were proposed for developing stochastic analysis with respect to fBm, using either (i) Malliavin calculus, or (ii) pathwise integration (which exploits the H\"older continuity of the paths). Pioneer works in this direction are: \cite{AMN01,decreusefond-ustunel98,nualart-rascanu02,zahle98}. Method (i) is relevant for the present article, and can be explained briefly as follows: we endow the space $\cE$ of linear combinations of indicator functions of the form $1_{[0,x]}$ with the inner product $\langle 1_{[0,x]},1_{[0,y]}\rangle_{\cP_0}=R_H(x,y)$, so that the map $1_{[0,x]} \mapsto B_x^{(H)} \in L^2(\Omega)$ becomes an isometry; then, the closure $\cP_0$ of $\cE$ is the domain of the Wiener integral with respect to $B^{(H)}$. In \cite{jolis10}, Jolis proved that $\cP_0$ coincides with the fractional Sobolev space $W^{\frac{1}{2}-H,2}(\bR)$, and therefore it is a space of distributions if $H>1/2$, and a space of functions if $H<1/2$. An alternative representation in \cite{DPV12} for the norm $\|\cdot\|_{\cP_0}$ was obtained in the case $H<1/2$, namely:
\begin{equation}
\label{Gagliardo}
\|\varphi\|_{\cP_0}^2=C_H \int_{\bR^2}|\varphi(x)-\varphi(y)|^2 |x-y|^{2H-2}dxdy,
\end{equation}
where $C_H=\frac{H(1-2H)}{2}$. Relation \eqref{Gagliardo} is called the {\em Gagliardo representation} and is very useful for problems in stochastic analysis, and in particular for the present article.

\medskip

SPDEs with colored noise in time have been considered for the first time in \cite{nualart-ouknine04}. Since then, this area has been growing steadily. However, the basic question of existence of solutions to (SHE) or (SWE) with colored (or fractional) noise in time is still an open problem in the case when the noise is multiplied by a Lipschitz function $\sigma(u)$ of the solution. The only case when it is known that these equations have unique solutions is the linear case, $\sigma(u)=u$. This case, which is known in the literature as the {\em parabolic Anderson model} (PAM) for the heat equation, respectively the {\em hyperbolic Anderson model} (HAM) for the wave equation, is studied using tools from Malliavin calculus, since the solution has an explicit Wiener chaos expansion.
This method was initiated by Hu and Nualart in \cite{hu-nualart09}. Various properties of the solution have been developed in \cite{CDOT,HHNT,HNS11} for (PAM), respectively \cite{BNQZ22,BQS19} for (HAM), to name just a few of the recent references.

In the present article, we consider the (HAM) driven by a Gaussian noise $\dot{W}$ which is white in time and fractional in space with Hurst index $H \in (\frac{1}{4},\frac{1}{2})$:
\begin{align}
\label{HAM}
	\begin{dcases}
		\dfrac{\partial^2 u}{\partial t^2} (t,x)
		= \dfrac{\partial^2 u}{\partial x^2} (t,x) + u(t,x) \dot{W}(t,x), &
		(t,x)\in \R_+\times \R,\\
		u(0,x) = 1, \ \dfrac{\partial u}{\partial t} (0,x) = 0, &\forall x\in \R.
	\end{dcases}
\end{align}

Formally, $\dot{W}$ is a zero-mean Gaussian noise with covariance
\[
\bE[\dot{W}(t,x) \dot{W}(s,y)]=\delta_{0}(t-s) \gamma(x-y), \quad
\mbox{where} \quad \gamma(x)=(|x|^{2H-2})''.
\]
Rigorously, $W=\{W(\varphi);\varphi \in \cD\}$ is a zero-mean Gaussian process defined on a complete probability space $(\Omega,\cF, \bP)$ and indexed by the set $\cD$ of infinitely differentiable functions on $\bR_{+}\times \bR$, with compact support. The covariance of $W$ is inspired by \eqref{spectral-cov} (in which we replace $1_{[0,x]}$ and $1_{[0,y]}$ by smooth functions), and is given by:
\begin{align}
\label{def_inn}
\bE[W(\varphi)W(\psi)]
=c_{H}\int_0^{\infty}\int_{\bR} \cF \varphi(t,\cdot)(\xi) \overline{\cF \psi(t,\cdot)(\xi)}|\xi|^{1-2H}d\xi dt \eqqcolon \langle \varphi,\psi\rangle_{\cH}.
\end{align}

We denote by $\cH$ the Hilbert space defined as the completion of $\cD$ with respect to the inner product $\langle \cdot, \cdot \rangle_{\cH}$. Then the map $\cD \ni \varphi \mapsto W(\varphi) \in L^2(\Omega)$ becomes an isometry which can be extended to $\cH$.
The process $\{W(\varphi);\varphi \in \cH\}$ is an isonormal Gaussian process, as defined in Malliavin calculus (see \cite{Nua06}), and $\cH$ is isomorphic to $L^2(\bR_{+};\cP_0)$.
Indicator functions of the form $1_{[0,t]\times [0,x]}$ lie in $\cH$ and the process $\{W_t(x)=W(1_{[0,t] \times [0,x]})\}_{x\in \bR}$ has the same distribution as $\sqrt{t}B^{(H)}$, since its covariance matches \eqref{spectral-cov}.

\medskip

A predictable process $u= \{u(t,x); t\geq 0,x \in \bR\}$ is a (mild) {\bf solution} to equation \eqref{HAM} if it satisfies the following integral equation:
\begin{equation}
\label{def-sol}
u(t,x)=1+\int_0^t \int_{\bR}G_{t-s}(x-y)u(s,y)W(ds,dy),
\end{equation}
where $G_t$ is the fundamental solution to the deterministic wave equation on $\bR_{+} \times \bR$:
\begin{equation}
\label{def-G}
G_t(x) \coloneqq \dfrac{1}{2} 1_{\{ | x| < t\}}.
\end{equation}
The integral on the right-hand side of \eqref{def-sol} is an It\^o integral, which coincides with the Skorohod integral, as observed in \cite{BJQ17}.

The existence and weak intermittency of the solution to \eqref{HAM} was proved in \cite{BJQ15}, respectively \cite{BJQ17}. In \cite{LHW21}, the existence of solution was studied in a more general scenario, where $u$ is replaced by a Lipschitz function $\sigma(u)$.
The existence of the solution to (PAM) with the same noise $\dot{W}$ as above was obtained in \cite{HHL+18}, while the exact asymptotic behaviour of its moments was established in \cite{HLN17}. The existence and H\"older continuity of the solution to (PAM) with space-time fractional noise of indices $H_0>1/2$ in time and $H<1/2$ in space was obtained in \cite{HL19}, under the condition $H_0+H>3/4$. The same problem for (HAM) was studied in \cite{SSX20} under the assumption $H \in (\frac{1}{4},\frac{1}{2})$. The (SHE) with the same noise and a general Lipschitz function $\sigma(u)$ multiplying the noise was studied in \cite{HHL+17}, under the restriction $\sigma(0)=0$; this condition was later removed in \cite{HW21}.

In all these references, the noise is spatially-homogeneous, i.e. it is invariant under translations. This property is transmitted to the solution $u$ in the form of strict stationarity of the process $\{u(t,x)\}_{x\in \bR}$. Without considerable effort, it is possible to prove that this process is also ergodic. The spatial ergodicity of the solution to an SPDE was proved for the first time in \cite{CKN+09} for the SHE with spatially homogeneous Gaussian noise (white noise in time), and a Lipshitz function $\sigma(u)$ multiplying the noise.

In the recent years, there has been a lot of interest in examining the asymptotic behaviour of the spatial average:
\[
F_{R}(t)=\int_{-R}^R \big(u(t,x)-1\big)dx.
\]
Since $\{u(t,x)\}_{x\in \bR}$ is strictly stationary and ergodic, by Brirkoff and von Neumann mean ergodic theorem,
the following law of large numbers holds:
\[
\frac{F_R(t)}{R} \to 0 \quad \mbox{a.s. and in $L^2(\Omega)$}, \quad \mbox{as $R \to \infty$}.
\]

A natural question is to investigate if $F_R(t)$ satisfies also a central limit theorem. For this, a novel technique was initiated in \cite{HNV20}, which
combines Stein's method for normal approximation with tools from Malliavin calculus. This method was originally developed for (SHE) with space-time white noise, and has been rapidly extended to other models. The paramount result is the {\em Quantitative Central Limit Theorem} (QCLT), which gives an estimate for the total variation distance $d_{TV}$ between $F_R(t)/\sigma_R(t)$ and a standard normal random variable $Z$, as a quantifier for the speed of convergence in distribution, when $R \to \infty$.
We recall that the total variation distance between random variables $X$ and $Y$ is given by:
\[
d_{\rm TV}(X,Y)=\sup_{B \in \cB(\bR)}|\bP(X \in B)-\bP(Y\in B)|,
\]
where $\cB(\bR)$ stands for the collections of all Borel subsets of $\bR$. The QCLT is closely related to a study of the order of magnitude of $\sigma_R^2(t)$, and can be extended to functional convergence.

\medskip

The following table summarizes the most important contributions to date, related to the problem of QCLT for solutions to SPDEs with spatially-homogeneous Gaussian noise, which can be white in time, fractional in time with index $H_0>1/2$, or time-independent. In some of these references, the temporal covariance of the noise can be more general, given by a non-negative-definite function $\gamma_0:\bR \to [0,\infty]$. For the sake of conciseness, we present only the fractional noise in time, when $\gamma_0(t)=|t|^{2H_0-2}$ with $H_0 \in (\frac{1}{2},1)$.
In this table, $\gamma(x)$ denotes the spatial covariance of the noise, $\sigma_R^2=\sigma_R^2(1)$, and $d_{TV}= d_{TV}(F_R/\sigma_R,Z)$ where $F_R=F_R(1)$. The notation $a_R \sim b_R$ indicates that $a_R/b_R \to C$ when $R \to \infty$, and $a_R\les b_R$ means that $a_R \leq C b_R$, where $C>0$ is a constant.

\begin{center}
\begin{tabular}{|c||c|c|} \hline
Noise           & (SHE) & (SWE) \\ \hline \hline
White & {\em Regular in space} & {\em Regular in space} \\
in time & $\bullet$ \cite{HNV20}: $d=1,\gamma=\delta_0\hspace{17mm}$ & $\bullet$ \cite{DNZ}: $d=1,\gamma(x)=|x|^{2H-2},H \in [\frac{1}{2},1)$  \\
($H_0=\frac{1}{2}$) & $\sigma_R^2 \sim R, d_{TV} \les R^{-1/2}$ & $\sigma_R^2 \sim R^{2H}, d_{TV} \les R^{H-1}$ \\
$\sigma$ arb. & $\bullet$ \cite{HNVZ}: $\gamma(x)=|x|^{-\beta},\beta\in (0,d)$ & $\bullet$ \cite{BNZ}: $d=2,\gamma(x)=|x|^{-\beta},\beta \in (0,2)\hspace{5mm}$ \\
& $\sigma_R^2 \sim R^{2d-\beta}, d_{TV} \les R^{-\beta/2}$ & $\sigma_R^2 \sim R^{4-\beta}, d_{TV} \les R^{-\beta/2}$ \\
&  & $\bullet$ \cite{NZ21}: $d\leq 2, \gamma\in L^1(\bR^d)\hspace{25mm}$ \\
& & $\sigma_R^2 \sim R^d, d_{TV} \les R^{-d/2}$ \\
& {\em Rough in space} \cite{NXZ22} & {\em Rough in space}\\
&  $d=1,H \in (\frac{1}{4},\frac{1}{2})$  & $d=1,H \in (\frac{1}{4},\frac{1}{2})$ \\
&  $\sigma_R^2 \sim R,d_{TV} \les R^{-1/2}$ & {\bf Open Problem 1}  \\ \hline
Fractional  & {\em Regular in space \cite{NXZ22,NZ20-1},} & {\em Regular in space \cite{BNQZ22}}  \\
in time & $\bullet$  $\gamma \in L^1(\bR^d)\hspace{30mm}$ & $\bullet$  $d\leq 2, \gamma \in L^1(\bR^d)\hspace{35mm}$  \\
$H_0 \in (\frac{1}{2},1)$ & $\sigma_R^2 \sim R^d, d_{TV} \les R^{-d/2}$ & $\sigma_R^2 \sim R^d, d_{TV} \les R^{-d/2}$ \\
$\sigma(u)=u$ & $\bullet$ $\gamma(x)=|x|^{-\beta}, \beta \in (0,d) \hspace{9mm}$ & $\bullet$   $d\leq 2,\gamma(x)=|x|^{-\beta}, \beta \in (0,d) \hspace{13mm}$ \\
& $\sigma_R^2 \sim R^{2d-\beta}, d_{TV} \les R^{-\beta/2}$ &
$\sigma_R^2 \sim R^{2d-\beta}, d_{TV} \les R^{-\beta/2}$ \\
& {\em Rough in space} \cite{NSZ20,NXZ22} & {\em Rough in space}  \\
&  $d=1,H<\frac{1}{2},H_0+H>\frac{3}{4}$  & $d=1,H \in (\frac{1}{4},\frac{1}{2})$ \\
&  $\sigma_R^2 \sim R,d_{TV} \les R^{-1/2}$ & {\bf Open Problem 2} \\
\hline
Time- & {\em Regular in space} \cite{BY23-1}  & {\em Regular in space}  \cite{BY22} \\
independent & $\bullet$ $\gamma \in L^1(\bR^d) \hspace{30mm}$ & $\bullet$ $d\leq 2, \gamma\in L^1(\bR^d) \hspace{33mm}$ \\
($H_0=1$) & $\sigma_R^2 \sim R^d, d_{TV} \les R^{-d/2}$ & $\sigma_R^2 \sim R^d, d_{TV} \les R^{-d/2}$ \\
$\sigma(u)=u$ & $\bullet$ $\gamma(x)=|x|^{-\beta},\beta \in (0,d)\hspace{8mm}$ & $\bullet$ $d\leq 2,\gamma(x)=|x|^{-\beta},\beta \in (0,d)\hspace{12mm}$ \\
 & {\em Rough in space} \cite{BY23-1} & {\em Rough in space} \cite{BY23-2} \\
 & $d=1,H \in (\frac{1}{4},\frac{1}{2})$ & $d=1,H \in (\frac{1}{4},\frac{1}{2})$ \\
 & $\sigma_R^2 \sim R, d_{TV}\les R^{-1/2}$ & $\sigma_R^2 \sim R, d_{TV}\les R^{-1/2}$ \\ \hline
\end{tabular}
\end{center}

In the case of the white noise in time (regular in space), reference \cite{CKNP22} gives the CLT for a function $g(u(t,x))$ of the spatial average of the solution to (SHE), while references \cite{CKNP-delta,KNP} studied the QCLT problem for a normalized version of the solution to (PAM) with delta initial condition. In the case of the (SWE) with white noise in time, the spatial ergodicity of the solution was proved in \cite{NZ20-2} for dimensions $d\leq 3$, while \cite{Ebi22} proved the convergence in the Wasserstein distance for the spatial average of the solution, in dimension $d=3$.

In this paper, we study the Open Problem 1 mentioned above, in the case $\sigma(u)=u$. We first show that the solution is strictly stationary and ergodic in the space variable, then we prove that $\sigma_R^2 \sim R$. The major effort is dedicated to the proof of QCLT. For this, we use the same method as in \cite{NXZ22} for (PAM), which relies on a second-order Poincar\'e inequality (Proposition 2.4 of \cite{NXZ22}). Due to the Gagliardo representation \eqref{Gagliardo} of the norm $\|\cdot\|_{\cP_0}$, we encounter the problem of estimating the fourth moments of the increments of the Malliavin derivative $Du(t,x)$, and of the rectangular increments of the second Malliavin derivative $D^2u(t,x)$. In the case of (PAM), these estimates are obtained in \cite{NXZ22} using highly non-trivial methods, that cannot be applied for (HAM).

Our method is simpler than that of \cite{NXZ22}, and relies on the following key relation
\begin{equation}
\label{Du}
D_{r,z}u(t,x)=u(r,z) v^{(r,z)}(t,x),
\end{equation}
between $Du$ and the solution $v=v^{(r,z)}$ of the integral equation (of Volterra type):
\begin{equation}
\label{eq-v}
v(t,x)=G_{t-r}(x-z)+\int_r^t \int_{\bR}G_{t-s}(x-y)v(s,y) W(ds,dy), \quad t\geq r,x\in \bR.
\end{equation}
Relation \eqref{Du} is valid for the solution $u$ of any SPDE with white noise in time, constant initial condition, and linear function $\sigma(u)=u$ multiplying the noise. But it may not hold for equations with colored noise in time, or time-independent noise. (Nevertheless, in the case of (HAM) with time-independent noise, it is still possible to derive an indirect argument relating $Du$ and $v$ which leads to estimates for the moments of $Du$ and its increments; see \cite[Theorem 4.1]{BY23-2}.) Key properties of $v^{(r,z)}$, which lie at the core of our methods, are: \\
(i) the moments of $v^{(r,z)}(t,x)$ are uniformly bounded in $(r,z,t,x)$ (see \cite[Example B.2]{BY23-2}); \\
(ii) $v^{(r,z)}$ satisfies the following relation (see the proof of Lemma \ref{lmm_vG} below):
\begin{align}
\label{special-G}
v^{(r,z)}(t,x)=2G_{t-r}(x-z)v^{(r,z)}(t,x);
\end{align}
(iii) the increments of $v^{(r,z)}$ have translation-invariance properties (see Lemma \ref{lmm_trans-v} below). \\
Properties (i) and (ii) may not hold for other models, such as (PAM); see \cite[Remark B.3]{BY23-2}.

\medskip

We are now ready to state the main results of the present article.

\begin{theorem}[Spatial ergodicity]
\label{EGD}
For any $t>0$, the process $\{u(t,x)\}_{x\in \bR}$ is strictly stationary and ergodic.
\end{theorem}

\begin{theorem}[Limiting covariance]
\label{limit-cov}
For any $t>0$ and $s>0$
\begin{equation}
\label{lim-cov}
\lim_{R \to \infty}\frac{\bE[F_R(t) F_R(s)]}{R}=K(t\wedge s),
\end{equation}
where
\begin{align*}
K(t)=4\pi \sum_{n\geq 2}C_H^n \int_{T_n(t)} \int_{\bR^{n-1}} & \prod_{j=1}^{n-1}
\frac{\sin^2((t_{j+1}-t_j)|\eta_j|)}{|\eta_j|^2} \\
& \times  \prod_{j=1}^{n-1}|\eta_j-\eta_{j-1}|^{1-2H} |\eta_{n-1}|^{1-2H} d \pmb{\eta}_{n-1} d \pmb{t}_n,
\end{align*}
with $\pmb{\eta}_{n-1} = (\eta_1,\dots, \eta_{n-1})$, $\pmb{t}_n=(t_1,\ldots,t_n)$ and by convention $\eta_0=0$. In particular, $\sigma_{R}^2(t)/R \to K(t)>0$ as $R \to \infty$.
\end{theorem}

\begin{theorem}[QCLT]
\label{QCLT}
For any $t>0$,
\[
d_{\rm TV}\left( \frac{F_R(t)}{\sigma_R(t)},Z\right) \leq C_{t,H}^{(2)} R^{-1/2},
\]
where $C_{t,H}^{(2)}>0$ is a constant depending on $t$ and $H$.
\end{theorem}

\begin{theorem}[Functional CLT]
\label{FCLT}
The process $\{R^{-1/2} F_R(t); t\geq 0\}$ converges in distribution as $R \to \infty$, to a zero-mean Gaussian process $\{\cG(t);t\geq 0\}$ with covariance
\[
\bE[\cG(t)\cG(s)]=K(t\wedge s).
\]
\end{theorem}

Theorem \ref{EGD} is derived from an ergodicity criterion credited to Maruyama \cite{Mar49}. By combining this criterion with the Gaussian-Poincar\'e inequality (see e.g., \cite{HP95, Nas58}), the spatial ergodicity of the solution to (SHE) was investigated in \cite{CKN+09}. For the solution to (SWE), the spatial ergodicity in dimensions $d\leq 3$ was proved in \cite{NZ20-2}. However, in the paper \cite{NZ20-2}, the spatial correlation function has to be a function (or a measure), rendering the result inapplicable in our case, when the correlation is a distribution. Consequently, new methods must be employed to establish the validity of Theorem \ref{EGD}.

The proof of Theorem \ref{QCLT} is based on the second-order Gaussian-Poincar\'{e} inequality (referenced as in e.g., \cite{Cha09,NPR09,Vid20}) following the strategy initially discovered in \cite{BNQZ22}, and further applied/developed in \cite{BY22,BY23-1,BY23-2,NXZ22}. It is important to note that several estimates for (PAM) with rough noise in space obtained in \cite{NXZ22} are precise but not essential for achieving the desired goal. Those estimates unnecessarily increase complexity and cannot be applied to (HAM). By thoroughly analyzing the implementation of the second-order Gaussian-Poincar\'e inequality in the spatial average of SPDE's, certain unnecessary computations were successfully eliminated in \cite{BY23-2}. This reduction in complexity significantly streamlined the proof and marked the final stage of the progress made in this paper.

\medskip

This article is organized as follows. In Section \ref{sec_prelim}, we include the preliminary results, leading to the estimates for Malliavin derivatives of the solution, which are essential for our developments. Section \ref{section-main} contains the proofs of Theorems \ref{EGD},\ref{limit-cov},\ref{QCLT} and \ref{FCLT}.

Throughout this paper, we denote by $\|\cdot\|_p$ the $L^p(\Omega)$-norm of a random variable on the probability space $(\Omega, \mathcal{F}, \mathbb{P})$. For any positive integer $n$, we make use of the notation $\pmb{x}_n \coloneqq (x_1,\dots, x_n)$ for an element in $\mathbb{R}^n$, and for any $t \in \R$, $T_n(t)$ stands for a subset of $[0,t]^n$: $T_n(t) \coloneqq \{\pmb{t}_n = (t_1,\ldots,t_n);0<t_1<\ldots<t_n<t\}$.

\section{Preliminary estimates}
\label{sec_prelim}

In this section, we present an overview of (HAM) driven by noise $W$ as above, including relevant preliminaries. We reference existing literature for certain results, while also introducing a few novel estimates supported by detailed proofs.

The next two identities will be used, whose proof follows from elementary calculus, and thus omitted. Let $t > 0$, then the following two identities hold. For any $\alpha \in (-1,1)$,
\begin{equation}
\label{sin-int}
\int_{\bR}\frac{\sin^2(t|\xi|)}{|\xi|^2}|\xi|^{\alpha}d\xi=
C_{\alpha} t^{1-\alpha},
\end{equation}
where $C_{\alpha}>0$ is a constant depending on $\alpha$. For any $\alpha_1>-1,\ldots,\alpha_n>-1$,
\begin{equation}
\label{beta-int}
\int_{T_n(t)}\prod_{j=1}^n(t_{j+1}-t_j)^{\alpha_j}d\pmb{t}_n =
\frac{\prod_{j=1}^{n}\Gamma(\alpha_j+1)}
{\Gamma(|\alpha|+n+1)}t^{|\alpha|+n},
\end{equation}
where we use the convention $t_{n+1}=t$, and we denote $|\alpha|= \sum_{j=1}^n \alpha_j$.

\subsection{Properties of the solution}
\label{section-exist}

In this section, we include some properties of the solution that will be needed in the sequel, including the spatial strict stationarity mentioned in Theorem \ref{EGD}.

As discussed in \cite{SSX20}, the solution to equation \eqref{HAM} has the Wiener chaos expansion:
\[
u(t,x)=1+\sum_{n\geq 1}I_n(f_n(\cdot,t,x)),
\]
where $I_n$ is the multiple integral with respect to $W$ and the kernel $f_n(\cdot,t,x)$ is given by:
\begin{align*}
f_n(\pmb{t}_n,\pmb{x}_n,t,x) = & f_n(t_1,\dots, t_n,x_1,\dots, x_n,t,x) \\
 \coloneqq  & G_{t-t_n}(x-x_n)\times \dots \times G_{t_2-t_1}(x_2-x_1)\mathbf{1}_{T_n(t)}(\pmb{t}_n)
\end{align*}
In this case,
\[
\bE|u(t,x)|^2=\sum_{n\geq 1}n!\|\widetilde{f}_n(\cdot,t,x)\|_{\cH_0^{\otimes n}}^2,
\]
where $\widetilde{f}_n(\cdot, t,x)$ is the symmetrization of $f_n(\cdot, t,x)$:
\[
\widetilde{f}_n(\pmb{t}_n,\pmb{x}_n,t,x)=\frac{1}{n!}\sum_{\rho \in S_n}f_{n}(t_{\rho(1)},\ldots,t_{\rho(n)},x_{\rho(1)},\ldots,x_{\rho(n)},t,x),
\]
with $S_n$ denoting the set of all permutations on $\{1,\dots, n\}$.

The following result was proved in \cite{BJQ15,BJQ17}; see \cite[Theorem 3.9]{BJQ15} or \cite[Theorem 3.5]{BJQ17}.

\begin{lemma}
\label{Lemma-u}
For any $H \in (1/4,1/2)$, equation \eqref{HAM} has a unique solution. Moreover, for any $p\geq 2$ and $ t \geq 0$,
\begin{equation}
\label{u-bound}
	\sup_{(r,x) \in [0,t] \times \bR} \|u(r,x)\|_p < C,
\end{equation}
where $C>0$ is a constant that depends on $(t,p,H)$.
\end{lemma}

The following result gives the spatial stationarity of $u$.

\begin{lemma}
\label{trans-u}
For any $t>0$, the process $\{u(t,x)\}_{x \in \bR}$ is strictly stationary. In particular, for any $x,h \in \bR$,
\[
u(t,x)-u(t,x+h) \stackrel{d}{=} u(t,0)-u(t,h).
\]
\end{lemma}

\begin{proof}
We show that for any $m \geq 1$ and for any $z_1,\ldots,z_m,h \in \bR$
\[
\big(u(t,z_1),\ldots, u(t,z_m)\big) \stackrel{d}{=}
\big(u(t,z_1+h),\ldots, u(t,z_m+h)\big).
\]
For any $j=1,\ldots,m$, the variable $u(t,z_j+h)$ has the chaos expansion:
\begin{align}
\label{diff-u}
u(t,z_j+h) = \sum_{n\geq 1}
\int_{T_n(t)\times \bR^n}& G_{t-t_n}(z_j+h-x_n)
 \prod_{i=1}^{n-1}G_{t_{i+1}-t_{i}}(x_{i+1}-x_{i})\nonumber\\
 & \times W(dt_1,dx_1)\ldots W(dt_n,dx_n),
\end{align}
where $T_n(t)=\{0<t_1<\ldots<t_n<t\}$.
Let $W^{(h)}=\{W^{(h)}(\varphi);\varphi \in \cD(\bR_{+} \times \bR)\}$ be the shifted noise, where
$W^{(h)}(\varphi)=W(\varphi(\cdot-h))$. Performing the formal change of variables $y_i=x_i-h$ for $i=1,\ldots,n$, we see that the multiple integral above is equal to
\begin{align*}
\int_{T_n(t) \times \bR^n} G_{t-t_n}(z_j-y_n) \prod_{i=1}^{n-1} G_{t_{i+1}-t_{i}}(y_{i+1}-y_{i}) W^{(h)}(dt_1,dy_1)\ldots W^{(h)}(dt_n,dy_n).
\end{align*}
Since $W$ is spatially homogeneous, $W^{(h)}\stackrel{d}{=}W$. Therefore, the vector of dimension $m$ whose $j$-th component is given by the series \eqref{diff-u} has the same distribution as the vector whose $j$-th component is $u(t,z_j)$.
\end{proof}

\begin{lemma}
\label{Lemma-difference of u}
For any $t>0$ and $p\ge 2$, we have
\begin{align*}
	\sup_{0\leq s \leq t}\int_{\bR} \big\|u(s,0)-u(s,h)\big\|_p^2 \, |h|^{2H-2} dh \le C,
\end{align*}
where $C>0$ is a constant that depends on $(t,H,p)$.
\end{lemma}

\begin{proof}
Denote by $C$ a constant that depends on $(t,H,p)$ and may be different in each of its appearances.
Using the uniform bound for $\|u(s,x)\|_p$ given by \eqref{u-bound}, we have:
\[
\int_{|h|>1}\|u(s,0)-u(s,h)\|_p |h|^{2H-2}dh\leq C \int_{|h|>1}|h|^{2H-2}dh<\infty.
\]
It remains to treat the integral over the set $|h|\leq 1$.
By using the chaos expansion \eqref{diff-u} for $u(t,0) - u(t,x)$,
and hypercontractivity (see e.g., \cite[Corollary 2.8.14]{NP12}), for any $p\geq 2$,
\begin{equation}
\label{u-diff-b}
\|u(s,0)-u(s,h)\|_p \leq \sum_{n\geq 1}(p-1)^{n/2}\|I_n({f}_{n,h}(\cdot,0,s))\|_2=
\sum_{n\geq 1}(p-1)^{n/2}[J_{n,h}(s)]^{1/2}
\end{equation}
where $J_{n,h}(s) \coloneqq n! \|{f}_{n,h}(\cdot,0,s)\|_{\cH^{\otimes n}}^2$.
It follows from \eqref{def_inn} that
\begin{align*}
J_{n,h}(s)&=c_H^n \int_{T_n(s)} \int_{\bR^n} \big|\cF f_{n,h}(\pmb{t}_n,\cdot,s,x)(\xi_1,\dots,\xi_n)\big|^2 \prod_{j=1}^{n}|\xi_j|^{1-2H} d\pmb{\xi}_n d\pmb{t}_n \\
&=c_H^n \int_{T_n(t)} \int_{\bR^n} \big| 1-e^{-i(\xi_1+\dots+\xi_n)h} \big|^2 \prod_{j=1}^{n}\big|\cF G_{t_{j+1}-t_j}(\xi_1+\dots+\xi_j)\big|^2 \prod_{j=1}^{n}
|\xi_j|^{1-2H}d\pmb{\xi}_n d\pmb{t}_n \\
&=c_H^n \int_{T_n(t)} \int_{\bR^n} \big| 1-e^{-i\eta_n h} \big|^2 \prod_{j=1}^{n}\big|\cF G_{t_{j+1}-t_j}(\eta_j)\big|^2 \prod_{j=1}^{n}
|\eta_j-\eta_{j-1}|^{1-2H}d\pmb{\eta}_n d\pmb{t}_n,
\end{align*}
where by convention $t_{n+1}=s$ and $\eta_0=0$.
Taking account of the (in)equalities $|a+b|^{1-2H}\leq |a|^{1-2H}+|b|^{1-2H}$ for all $a, b \in \R$, and
\[
x_1 \prod_{j=2}^n (x_j+x_{j-1})=\sum_{\pmb{\alpha}_n \in A_n}x_j^{\alpha_j}, \quad \text{for all } x_1,\dots, x_n \in \R_+,
\]
where $A_n$ is a set of multi-indices $\pmb{\alpha}_n=(\alpha_1,\ldots,\alpha_n)$ such that $\alpha_1\in \{1,2\}$, $\alpha_n\in \{0,1\}$, $\alpha_2,\ldots,\alpha_{n-1} \in \{0,1,2\}$, $\sum_{j=1}^{n}\alpha_j=n$ and $\alpha_j+\alpha_{j-1}\in \{1,2,3\}$ for any $j=1,\ldots,n-1$. Then the cardinality ${\rm card}(A_n) = 2^{n-1}$ (see \cite[Page 7]{SSX20}). Hence
\begin{align}
\label{prod}
	\prod_{j=1}^{n}|\eta_j-\eta_{j-1}|^{1-2H}\leq \sum_{\pmb{\alpha}_n\in D_n}\prod_{j=1}^n |\eta_j|^{\alpha_j},
\end{align}
where $D_n$ is a subset of $A_n$ fulfilling $\alpha_{j} = (1-2H)\alpha_j$, $j=1,\dots,n$.
It follows that
\begin{align*}
 J_{n,h}(s) \leq c_H^n \sum_{\pmb{\alpha}_n \in D_n} \int_{T_n(t)}&
 \prod_{j=1}^{n-1}\left(\int_{\bR}|\cF G_{t_{j+1}-t_j}(\eta_j)|  |\eta_j|^{\alpha_j} d\eta_j \right)\\
 &\times \left(\int_{\bR} |1-e^{-i\eta_n h}|^2 |\cF G_{t-t_n}(\eta_n)|^2 |\eta_n|^{\alpha_n}d\eta_n \right) d\pmb{t}_n
 \end{align*}
As a result of identity \eqref{sin-int} and the fact that $|\eta_n|^{\alpha_n}\leq 1+|\eta_n|^{1-2H}$ (since $\alpha_n\in \{0,1-2H\}$),
\begin{align*}
J_{n,h}(s) & \leq C^n \sum_{\alpha \in A_n} \int_0^t
\left( \int_{\bR}|1-e^{-i\eta_n h}|^2  |\cF G_{t-t_n}(\eta_n)|^2 (1+|\eta_n|^{1-2H})d\eta_n\right) \\
& \qquad \qquad \qquad  \left(\int_{T_{n-1}(t_n)}
\prod_{j=1}^{n-1}(t_{j+1}-t_j)^{1-\alpha_j} d \pmb{t}_{n-1} \right)dt_n
\end{align*}
Due to inequality \eqref{beta-int}.
and the fact that $\Gamma(an+b+1) \geq c_a^n (n!)^a$ for any $a>0$ and $b\in \bR$ with some constant $c_a$ depending only on $a$, it can be proved that
\[
\int_{T_{n-1}(t_n)}
\prod_{j=1}^{n-1}(t_{j+1}-t_j)^{1-a_j} d \pmb{t}_{n-1} \leq \frac{C^n}{(n!)^{2H+1}} t_n^{(2H+1)n-2}.
\]
Therefore,
\begin{align*}
J_{n,h}(s) \leq &  \frac{C^n}{(n!)^{2H+1}} \int_0^t \int_{\bR}
|1-e^{-i\eta_n h}|^2  (1+|\eta_n|^{1-2H}) \frac{\sin^2((t-t_n)|\eta_n|)}{|\eta_n|^2} t_n^{(2H+1)n-2}d\eta_n dt_n\\
\leq & \frac{C^n}{(n!)^{2H+1}} s^{(2H+1)n-1}
\int_{\bR}
|1-e^{-i\eta_n h}|^2 (1+|\eta_n|^{1-2H}) \frac{1}{|\eta_n|^2} d\eta_n,
\end{align*}
where the last inequality follows from the bound
$\sin^2((s-t_n)|\eta_n|) \leq 1$.
Using the fact that $|1-e^{-ix}|^2 =2(1-\cos x)$ and the identity
\[
\int_{\bR}(1-\cos(\xi x)) |x|^{-\alpha-1}dx=C_{\alpha}|\xi|^{\alpha}, \quad \mbox{for any $\alpha \in (0,2)$},
\]
we obtain that
$J_{n,h}(s) \leq \frac{C^{n}}{(n!)^{2H+1}} \big(|h|+|h|^{2H}\big)$.
Plugging this inequality to \eqref{u-diff-b}, we get
\[
\|u(s,0)-u(s,h)\|_p \leq (|h|^{1/2}+|h|^H)\sum_{n\geq 1}(p-1)^{n/2}\frac{C^{n}}{(n!)^{H+1/2}} \eqqcolon C (|h|^{1/2}+|h|^{H}).
\]
This implies
\begin{align*}
\int_{|h|\leq 1}\|u(s,x)-u(s,x+h)\|_p^2 |h|^{2H-2}dh \leq C \int_{|h|\leq 1}  \big(|h|^{1/2}+|h|^{H}\big)^2 |h|^{2H-2}dh < \infty,
\end{align*}
because $H>1/4$. The proof of this lemma is complete.
\end{proof}

\subsection{(HAM) with delta initial velocity}
\label{section-delta}

In this section, we study equation \eqref{eq-v}.
We begin by recalling a basic fact which states that the solution to the deterministic wave equation:
\[
\frac{\partial^2 w}{\partial^2 t}(t,x)= \frac{\partial^2 w}{\partial^2 x}(t,x), \quad t>r,x\in \bR
\]
with initial condition $w(r,x)=u_0(x)$ and $\frac{\partial w}{\partial t}(r,x)=v_0(x)$ is given by:
\[
w(t,x)=(G_{t-r}*v_0)(x)+\frac{\partial}{\partial t}(G_{t-r}*u_0)(x).
\]
When $u_0=0$ and $v_0=\delta_z$ (for fixed $z \in \bR$), this becomes $w(t,x)=G_{t-r}(x-z)$. This observation motivates the definition below.

Consider the following model:
\begin{align}
\label{ham-delta}
	\begin{dcases}
		\dfrac{\partial^2 v}{\partial t^2} (t,x)
		= \dfrac{\partial^2 v}{\partial x^2} (t,x) + v(t,x) \dot{W}(t,x), \
		t>r, \ x \in \bR\\
		v(r,\cdot) = 0, \ \dfrac{\partial v}{\partial t} (r,\cdot) = \delta_z,
	\end{dcases}
\end{align}
We say that a random field $v^{(r,z)} = \{v^{(r,z)}(t,x) ; (t,x) \in [r,\infty) \times \R \}$ is a (mild) {\bf solution} of equation \eqref{ham-delta} if it satisfies the integral equation \eqref{eq-v}.

It can be proved that equation \eqref{ham-delta} has a unique solution (see e.g., \cite[Example B.2]{BY23-2}). Moreover, this solution has the chaos expansion:
\begin{equation}
\label{chaos-v}
v^{(r,z)}(t,x) = G_{t-r}(x-z) + \sum_{n\geq 1}I_n(g_n(\cdot,r,z,t,x)),
\end{equation}
where $g_0(r,z,t,x)=G_{t-r}(x-z)$ and for $n\geq 1$,
\begin{equation}
\label{def-gk}
	g_n(\pmb{t}_n,\pmb{x}_n,r,z,t,x)
	= G_{t-t_{n}}(x-x_{n}) \times \dots \times G_{t_1-r}(x_1-z).
\end{equation}
In addition, for any $p\geq 2$ and $t>0$,
\begin{equation}
\label{v-bounded}
\sup_{\substack{0\leq r \leq s \leq t \\ x,z\in \bR}} \big\|v^{(r,z)}(s,x)\big\|_p < C,
\end{equation}
where $C > 0$ is a constant depends on $(t,p,H)$.

\begin{lemma}\label{lmm_vG}
For any $0 \leq r \leq s \leq t < \infty$, $x,z \in \R$, and $p \geq 2$, we have:
\begin{equation}
\label{vG}
\big\|v^{(r,z)}(s,x)\big\|_p =2G_{s-r}(x-z)\big\|v^{(r,z)}(s,x)\big\|_p \leq C G_{s-r}(x-z).
\end{equation}
where $C>0$ is a constant that depends on $(t,p,H)$.
\end{lemma}
\begin{proof}
Taking account of the special form \eqref{def-G} of $G$, we see that $g_n(\cdot,r,z,t,x)=2G_{t-r}(x-z) g_n(\cdot,r,z,t,x)$. From this, we deduce using the chaos expansion \eqref{chaos-v} that \eqref{special-G} holds.
Thus, \eqref{vG} follows from \eqref{v-bounded} and \eqref{special-G}. The proof of this lemma is complete.
\end{proof}

The next lemma presents some translation-invariance properties of $v^{(r,z)}$.

\begin{lemma}\label{lmm_trans-v}
	For any $0 \leq r \leq t<\infty$ and $x,x',z,z' \in \bR$, the following properties hold:
	\begin{enumerate}[{\rm a)}]
		\item \label{trans-v0'}	$\displaystyle v^{(r,z)}(t,x)\stackrel{d}{=}v^{(r,0)}(t,x-z)$.

		\item \label{trans-v1'} $\displaystyle v^{(r,z)}(t,x)-v^{(r,z')}(t,x') \stackrel{d}{=}
	v^{(r,0)}(t,x-z)-v^{(r,z'-z)}(t,x'-z)$.

	\item \label{trans-v2} $\begin{aligned}[t]
& v^{(r,z)}(t,x)-v^{(r,z)}(t,x')-
v^{(r,z')}(t,x)+v^{(r,z')}(t,x')
\stackrel{d}{=}\\
& \quad \quad \quad
v^{(r,0)}(t,x-z)-v^{(r,0)}(t,x'-z)-
v^{(r,z'-z)}(t,x-z)+v^{(r,z'-z)}(t,x'-z).
\end{aligned}$	\end{enumerate}
\end{lemma}

\begin{proof}
a) This is similar to the proof of Lemma \ref{trans-u}.

b) We have the following chaos expansion:
\begin{align}
\nonumber
& v^{(r,z)}(t,x)-v^{(r,z')}(t,x') = G_{t-r}(x-z)-G_{t-r}(x'-z')+\\
\nonumber
& \quad \quad \quad \sum_{n=1}^{\infty} \int_{r<t_1<\ldots<t_n<t} \int_{\bR^n}\Big( G_{t-t_n}(x-x_n) G_{t_1-r}(x_1-z) - G_{t-t_n}(x'-x_n) G_{t_1-r}(x_1-z')  \Big) \\
\label{diff-v}
& \qquad \qquad \qquad \qquad \qquad \prod_{i=1}^{n-1}G_{t_{i+1}-t_i}(x_{i+1}-x_i) W(dt_1,dx_1)\ldots W(dt_n,dx_n).
\end{align}
Performing the formal change of variables $y_i=x_i-z$ for $i=1,\ldots,n$, we see that the multiple integral above is equal to
\begin{align*}
& \int_{r<t_1<\ldots<t_n}\int_{\bR^n} \Big( G_{t-t_n}(x-z-y_n) G_{t_1-r}(y_1) - G_{t-t_n}(x'-z-y_n) G_{t_1-r}(y_1-z'+z)  \Big) \\
& \qquad \qquad \qquad \prod_{i=1}^{n-1}G_{t_{i+1}-t_i}(y_{i+1}-y_i) W^{(z)}(dt_1,dy_1)\ldots W^{(z)}(dt_n,dy_n).
\end{align*}
Since $W \stackrel{d}{=}W^{(z)}$, the series \eqref{diff-v} has the same distribution as the series which gives the chaos expansion of $v^{(r,0)}(t,x-z)-v^{(r,z'-z)}(t,x'-z)$.

c) The chaos expansion of the term on the left-hand side is:
\begin{align*}
&  G_{t-r}(x-z)-G_{t-r}(x'-z) -
G_{t-r}(x-z')+G_{t-r}(x'-z')+\\
& \quad \sum_{n\geq 1}\int_{r<t_1<\ldots<t_n<t} \int_{\bR^n}
\big( G_{t-t_n}(x-x_n)-G_{t-t_n}(x'-x_n)  \big)
\prod_{i=1}^{n-1}G_{t_{i+1}-t_{i}}(x_{i+1}-x_{i})\\
& \quad \quad \quad \quad \quad \quad \quad
\big( G_{t_1-r}(x_1-z)-G_{t_1-r}(x_1-z') \big)
W(dt_1,dx_1) \ldots W(dt_n,dx_n)
\end{align*}
Then we perform the formal change of variables $y_i=x_i-z$ for $i=1,\ldots,n$ in the multiple integral of order $n$.
The conclusion follows using the fact that $W\stackrel{d}{=}W^{(z)}$.
\end{proof}

\begin{lemma}
\label{Lemma-v}
	For any $p\ge 2$,  $q>0$ and $ t \in \R_+$, we have
\begin{align*}
	\sup_{\substack{0 \leq r \leq s \leq t\\z \in \R}}\int_{\bR} \big\|v^{(r,z)}(s,x') \big\|_p^q dx' +
	\sup_{\substack{0 \leq r \leq s \leq t\\x \in \R}}\int_{\bR} \big\|v^{(r,z')}(s,x)\big\|_p^q dz'  \le C,
\end{align*}
where $C>0$ is a constant that depends on $(t, p ,q, H)$.
\end{lemma}

\begin{proof}
This is a direct consequence of Lemma \ref{lmm_vG}.
\end{proof}

The properties listed in the following lemma have been proved in \cite{BY23-2}: parts a) and c) correspond to relations (98)--(99), ibid. and part b) was shown in the proof of Lemma 5.4c), ibid. Note that $v^{(r,z)}$ is denoted by $V_1^{(r,z)}$ in \cite{BY23-2}.

\begin{lemma}
\label{lem-v}
For any $p\ge 2$ and $t \in \R_+$, we have:
\begin{enumerate}[{\rm a)}]
\item \label{lem-va} $ \displaystyle  \sup_{0 \leq r \leq s \leq t}\int_{\bR^2}\left\| v^{(r,0)}(s,x) - v^{(r,h)}(s,x) \right\|_p^2 |h|^{2H-2}dxdh \leq C,$

\item \label{lem-vb} $\displaystyle \sup_{0 \leq r \leq s \leq t}\int_{\bR^2}\left\| v^{(r,0)}(s,x) - v^{(r,0)}(s,x+h) \right\|_p^2 |h|^{2H-2}dxdh \leq C,$

\item \label{lem-vc} $\begin{aligned}[t]
 \sup_{0 \leq r \leq s \leq t }\int_{\bR^3}  & \left\| v^{(r,0)}(s,x)-v^{(r,h)}(s,x) - v^{(r,0)}(s,x+\hbar)+v^{(r,h)}(s,x+\hbar) \right\|_p^2\\
 &\times  |h|^{2H-2}|\hbar|^{2H-2} dx dh d\hbar \le C,
\end{aligned}$
\end{enumerate}
where $C>0$ is a constant that depends on $(t,p,H)$.
\end{lemma}

\begin{lemma}
\label{Lemma-difference of v}
For any $p\ge 2$ and $t \in \R_+$, we have:
\begin{enumerate}[{\rm a)}]
	\item \label{lmm-diff-va} $\displaystyle \sup_{0 \leq r \leq s \leq t }\int_{\bR} \left( \int_{\bR} \big\| v^{(r,0)}(s,x)-v^{(r,h)}(s,x) \big\|_p dx \right)^2 |h|^{2H-2} dh \le C,$

	\item \label{lmm-diff-vb} $\displaystyle \sup_{0 \leq r \leq s \leq t }\int_{\bR} \left( \int_{\bR} \big\| v^{(r,0)}(s,x)-v^{(r,0)}(s,x+h) \big\|_p dx \right)^2 |h|^{2H-2} dh \le C,$

	\item \label{lmm-diff-vc} $\begin{aligned}[t]
\sup_{0 \leq r \leq s \leq t } \int_{\bR^2} & \left( \int_{\bR} \big\| v^{(r,0)}(s,x)-v^{(r,h)}(s,x) - v^{(r,0)}(s,x+ \hbar)+v^{(r,h)}(s,x+\hbar) \big\|_p dx \right)^2 \\
&\times |h|^{2H-2}|\hbar|^{2H-2} dh d\hbar \le C,
	\end{aligned}$
\end{enumerate}
where $C>0$ is a constant that depends on $(t,p,H)$.
\end{lemma}

\begin{proof}
a) The proof follows from the same argument as in that of \cite[Lemma 5.5]{BY23-2}, with minor modifications. Denote by $I_1$ and $I_2$ the integral in \ref{lmm-diff-va}) on the region $\{h ; |h| > 1\}$ and $\{h ; |h| \leq 1\}$, respectively.
Then, by Lemma \ref{Lemma-v},
\begin{align*}
\sup_{0\leq r \leq s \leq t}\int_{\bR} \left\| v^{(r,0)}(s,x) - v^{(r,h)}(s,x) \right\|_p dx
\leq &
\sup_{0\leq r \leq s \leq t}\int_{\bR} \Big(\left\| v^{(r,0)}(s,x)\right\|_p + \left\| v^{(r,h)}(s,x) \right\|_p dx \Big)\\
 \leq & C,
\end{align*}
 and hence
 \[
 I_1 \leq C \int_{|h|>1}|h|^{2H-2} dh = C.
 \]
Using \eqref{special-G}, we have
$I_2 \leq  C \big(I_{2,1}+I_{2,2}\big)$, where
\begin{align*}
I_{2,1} & \coloneqq \sup_{0\leq r \leq s \leq t}\int_{|h|\leq 1}\left(  \int_{\bR}
G_{s-r}(x) \left\| v^{(r,0)}(s,x) - v^{(r,h)}(s,x) \right\|_p dx   \right)^2 |h|^{2H-2} dh\\
I_{2,2}& \coloneqq \sup_{0\leq r \leq s \leq t} \int_{|h|\leq 1}\left(  \int_{\bR}
\big|G_{s-r}(x) - G_{s-r}(x-h)\big| \left\| v^{(r,h)}(s,x) \right\|_p dx   \right)^2 |h|^{2H-2} dh.
\end{align*}	
As a result of the Cauchy-Schwarz inequality and the fact that $\|G_{s-r}\|_{L^2(\R)}^2 = (s-r)/2$,
\[
I_{2,1} \leq C \sup_{0\leq r \leq s \leq t}\int_{\bR^2}\left\| v^{(r,0)}(s,x) - v^{(r,h)}(s,x) \right\|_p^2 |h|^{2H-2}dxdh.
\]
The last integral is uniformly bounded, by Lemma \ref{lem-v}\ref{lem-va}).
Using \eqref{v-bounded} and the fact that $\|G_{s-r}(\cdot) - G_{s-r}(\cdot -h)\|_{L^1(\R)} \leq |h|$, we have:
\begin{align*}
		I_{2,2}	& \leq  C \sup_{0\leq r \leq s \leq t}\int_{|h|\leq 1} \left( \int_{\bR} \big|G_{s-r}(x) - G_{s-r}(x-h)\big|  dx  \right)^2  |h|^{2H-2}dh \leq C \int_{|h|\leq 1} |h|^{2H} dh=C.
\end{align*}

 b) The proof is similar to that of \ref{lmm-diff-va}), and thus skipped.

c)
Denote by $J$ the integral in \ref{lmm-diff-vc}). Then,
\[
J \leq C \big(J_1 + J_2 + J_3 + J_4\big),
\]
where
\[
J_k \coloneqq \sup_{0\leq r \leq s \leq t}\int_{\bR^2} \left( \int_{\bR}F_k(x,h,\hbar)dx\right)^2 |h|^{2H-2} |\hbar|^{2H-2} dh d\hbar,\quad k = 1, \dots, 4,
\]
with
\begin{gather*}
F_1 \coloneqq 2 G_{s-r}(x) \big\|v^{(r,0)}(s,x) - v^{(r, h)}(s,x) - v^{(r,0)}(s,x + \hbar) + v^{(r,h)}(s,x + \hbar)\big\|_p,\\
F_2 \coloneqq 2|G_{s-r}(x) - G_{s-r}(x + \hbar )| \big\|v^{(r,0)}(s,x + \hbar) - v^{(r, h)}(s,x + \hbar)\big\|_p,\\
F_3 \coloneqq  2|G_{s-r}(x) - G_{s-r}(x - h) - G_{s-r}(x + \hbar ) + G_{s-r}(x + \hbar  - h)|\big\|v^{(r, h)}(s,x)\big\|_p,
\shortintertext{and}
F_4 \coloneqq 2 |G_{s-r}(x + \hbar) - G_{s-r}(x + \hbar - h)| \big\| v^{(r, h)}(s,x) - v^{(r, h)}(s,x + \hbar)\big\|_p.
\end{gather*}
The conclusion follows from the same argument as in the proof of \cite[Lemma 5.5e)]{BY23-2}, by using Lemmas \ref{lem-v} and \ref{Lemma-v}, as well as relations (58) and (59), ibid. The proof of this lemma is complete.
\end{proof}

\subsection{Malliavin derivatives and (HAM) with the delta initial velocity}
\label{section-Malliavin}

In this section, we establish the connection between the Malliavin derivatives $Du$ and $D^2 u$ and the solution $v^{(r,z)}$ of \eqref{ham-delta}.

\begin{lemma}\label{lmm-dprod}
Let $u$ be the solution to \eqref{HAM}, and let $v^{(r,z)}$ be the solution to \eqref{ham-delta} with arbitrary $(r,z) \in \R_+\times \R$. Then,
\begin{enumerate}[{\rm a)}]
\item \label{D-prod}
$\displaystyle
D_{r,z}u(t,x)=u(r,z) v^{(r,z)}(t,x)$, for all $0 \leq r \leq t <\infty$ and $x,z \in \R$.

\item \label{D2-prod}
$\displaystyle D_{(\theta,w),(r,z)}^2 u(t,x)=u(\theta,w)v^{(\theta,w)}(r,z) v^{(r,z)}(t,x)$ for all $0 \leq \theta \leq r \leq t <\infty $ and $x,z,w \in \R$.
\end{enumerate}
\end{lemma}
\begin{proof}
We only provide the proof of \ref{D-prod}). One can show \ref{D2-prod}) similarly. Note that the Malliavin derivative $D_{r,z}u(t,x)$ has the chaos expansion: for any $(r,z) \in [0,t]\times \R$,
\begin{equation}
\label{chaos-D}
D_{r,z}u(t,x)=\sum_{n\geq 1}n I_{n-1}(\bar{f}_n(\cdot,r,z,t,x))=\sum_{n\geq 1} \sum_{j=1}^{n}I_{n-1}(f_j^{(n)}(\cdot,r,z,t,x)),
\end{equation}
where
\begin{align*}
f_j^{(n)}(  \pmb{t}_{n-1}  , \pmb{x}_{n-1},r,z,t,x) = &
f_n(t_1,\ldots,t_{j-1},r,t_{j},\dots, t_{n-1},x_1,\dots, x_{j-1}, z, x_{j},\ldots,
x_{n-1},t,x)\\
= &
G_{t-t_{n-1}}(x-x_{n-1})\times \dots \times G_{t_j-r}(x_j-z)G_{r-t_{j-1}}(z-x_{j-1})\\
& \times \dots \times G_{t_2-t_1}(x_2-x_1) \mathbf{1}_{T_{n-1}^{j}(t,r)}(\pmb{t}_{n-1}),
\end{align*}
with
\begin{align*} T_{n-1}^{j}(t,r) \coloneqq \{ \pmb{t}_{n-1} \in [0,t]^{n-1} ; 0<t_1<\ldots<t_{j-1}<r<t_j<\ldots<t_{n-1}<t\}.
\end{align*}
Note that the function $f_j^{(n)}(\cdot,r,z,t,x)$ can be written as follows:
\begin{align}\label{decomp-fjn}
f_j^{(n)}(  \pmb{t}_{n-1}  , \pmb{x}_{n-1},r,z,t,x)
= & f_{j}(\pmb{t}_{j - 1}, \pmb{x}_{j - 1},r,z) g_{n-j}(\pmb{t}_{j:n-1}, \pmb{x}_{j;n-1},r,z,t,x),
\end{align}
where $g_{n-j}$ is given by \eqref{def-gk} with $g_0(r,z,t,x) \coloneqq G_{t-r}(x-z)$, $\pmb{t}_{j:n-1} \coloneqq (t_{j},\dots, t_{n-1})$, and $\pmb{x}_{j:n-1} \coloneqq (x_{j},\dots, x_{n-1})$.
Because the two functions appearing in this decomposition have ``disjoint temporal supports'' (see the footnote on page 16 of \cite{BNQZ22}) and the noise is white in time, it follows
\[
f_{j - 1} (\cdot, r, z) \otimes_{k} g_{n-j}(\cdot, r,z,t,x) \equiv 0, \quad \text{for all } k = 1,\dots, (j-1) \wedge (n-j).
\]
As a result, using the product formula e.g., \cite[Proposition 1.1.3]{Nua06}, we have
\[
I_{n-1}\big(f_j^{(n)}(\cdot,r,z,t,x)\big) =
 I_{j-1}\big(f_{j-1}(\cdot,r,z)\big)  I_{n-j} \big(g_{n-j}(\cdot,r,z,t,x)\big),
\]
Interchanging the order of summation in \eqref{chaos-D},
we obtain:
\begin{align*}
D_{r,z}u(t,x) = & \sum_{j\geq 1} \sum_{n\geq j} I_{n-1}(f_j^{(n)}\big(\cdot,r,z,t,x)\big) \\
 = & \sum_{j\geq 1} I_{j-1} \big(f_{j-1}(\cdot,r,z)\big) \sum_{n\geq j} I_{n-j} \big(g_{n-j} (\cdot,r,z,t,x)\big).
\end{align*}
This justifies \ref{D-prod}). The proof of this lemma is complete.
\end{proof}

\subsection{Moment estimates for the Malliavin derivatives}

In this section, we derive some estimates for the moments of the first and second Malliavin derivatives of $u(t,x)$.

Notice that $u$ is adapted and the noise has independent increments in time. Therefore, $u(r,z)$ is independent of $v^{(r,z)}$. Hence, using \eqref{u-bound}, for any $p\geq 2$, $r \in [0,t]$ and $x,z \in \bR$,
\[
\|D_{r,z}u(t,x)\|_p = \|u(r,z)\|_{p} \|v^{(r,z)}(t,x)\|_{p} \leq C_t \|v^{(r,z)}(t,x)\|_{p}.
\]
A similar deduction can be done for $D^2 u(t,x)$ and thus taking account of Lemma \ref{lmm_vG}, we can summarize the next proposition.
\begin{proposition}
Let $u$ be the solution to \eqref{HAM}, then for any $0 \leq r  \leq s \leq t < \infty$ and $x,y,z \in \R$, the next inequalities hold:
\begin{equation}
\label{D-bound}
\|D_{r,z}u(t,x)\|_p \leq C_t G_{t-r}(x-z),
 \end{equation}
 and
 \begin{equation}
\label{D2-bound}
\|D_{(r,z)(s,y)}^2 u(t,x)\|_p \leq
		C_tG_{t-s}(x-y)G_{s-y}(y-z),
\end{equation}
where $C_t > 0$ depending only on $t$.
\end{proposition}
The next proposition about the estimates for the increments of the Malliavin derivatives of $u(t,x)$ will be used in the proof of the main results. To this end, we should first introduce the following notation. For any $r,s,t \in \R_+$ such that $0 \leq r \vee s \leq t$, and $x,y,y',z, z' \in \R$, we denote
\begin{align}\label{def_delta}
\Delta_{h}(r,z,t,x) \coloneqq D_{r, z + h} u(t,x) - D_{r,z} u(t,x),
\end{align}
and
\begin{align}\label{def_rectg}
\square_{h,\hbar}(r, z, s, y, t, x) \coloneqq & D_{(r,z+h),(s,y+\hbar)}^2 u(t,x) -
D_{(r,z),(s,y+\hbar)}^2 u(t,x) \nonumber \\
&  - D_{(r,z+h),(s,y)}^2    u(t,x) + D_{(r,z), (s,y)}^2 u(t,x).
\end{align}

\begin{proposition}\label{prop_incre}
For any $t \in \R_+$, and $p\ge 2$, we have:
\begin{enumerate}[{\rm a)}]
\item \label{prop_incre-a} $\displaystyle \sup_{0\leq r \leq t}\int_{\R} \sup_{z\in \R} \Big(\int_{\R}  \big\|\Delta_{h}(r,z,t,x)\big\|_p dx \Big)^2 |h|^{2H - 2} d h \leq C$,

\item \label{prof_incre-d} $\displaystyle  \sup_{\substack{0\leq r \leq t\\ x' \in \R}}\int_{\R^3}   \big\|\Delta_{h}(r,z,t,x)\big\|_p \big\|\Delta_{h}(r,z,t,x')\big\|_p |h|^{2H - 2} d h dz dx \leq C$,

\item \label{prop_incre-b} $\displaystyle \sup_{0\leq r\vee s \leq t}\int_{\R^2} \sup_{y \in \R}\Big(\int_{\R^2}  \big\| \square_{h,\hbar}(r,z,s,y,t,x) \big\|_p dx dz \Big)^2 |h|^{2H - 2} |\hbar|^{2H - 2} dh d \hbar \leq C$,

\item \label{prop_incre-c} $\displaystyle \sup_{\substack{0\leq r \vee s \leq t\\ x \in \R}}\Big(\int_{\R^2}   \big\| \square_{h,\hbar}(r,z,s,y,t,x) \big\|_p dy dz \Big)^2 |h|^{2H - 2} |\hbar|^{2H - 2} dh d \hbar < C$,
\end{enumerate}
where $C > 0$ is a constant that depends on $(t,p,H)$.
\end{proposition}
\begin{proof}
a) By using Lemma \ref{lmm-dprod}\ref{D-prod}) , we can write
\begin{align*}
\Delta_{h}(r,z,t,x) =  & u(r,z + h) v^{(r,z + h)}(t,x)-
u(r,z)v^{(r,z)}(t,x) \\
= & \big(u(r,z + h)-u(r,z))v^{(r,z)}(t,x)+
u(r,z + h)\big(v^{(r,z + h)}(t,x)-v^{(r,z)}(t,x) \big).
\end{align*}
Then, it follows from the Cauchy-Schwartz inequality and Lemmas \ref{trans-u} and \ref{lmm_trans-v}\ref{trans-v1'}) that
\begin{align}\label{eq_deltahcs}
\big\|\Delta_{h}(r,z,t,x)\big\|_p \leq & \big\| u(r,h)-u(r,0)\|_{2p}\|v^{(r,z)}(t,x)\|_{2p}\nonumber\\
& +
\|u(r,z + h)\|_{2p}\big\| v^{(r,h)}(t,x-z)-v^{(r,0)}(t,x - z) \big\|_{2p}.
\end{align}
Hence, Proposition \ref{prop_incre}\ref{prop_incre-a}) is a consequence of Lemmas \ref{Lemma-u}, \ref{Lemma-difference of u}, \ref{Lemma-v}, and \ref{Lemma-difference of v}\ref{lmm-diff-vb}).

\medskip

b) Using \eqref{eq_deltahcs}, we can write
\begin{align*}
	\sup_{\substack{0\leq r \leq t\\ x' \in \R}} \int_{\R^3}   \big\|\Delta_{h}(r,z,t,x)\big\|_p \big\|\Delta_{h}(r,z,t,x')\big\|_p |h|^{2H - 2} d h dz dx = K_1 + K_2 + K_3 + K_4,
\end{align*}
where
\begin{align*}
	K_1 \coloneqq \sup_{\substack{0\leq r \leq t\\ x' \in \R}} \int_{\R^3} &  \big\| u(r,h)-u(r,0)\|_{2p}^2\|v^{(r,z)}(t,x)\|_{2p} \|v^{(r,z)}(t,x')\|_{2p} |h|^{2H - 2} d h dz dx,\\
	K_2 \coloneqq \sup_{\substack{0\leq r \leq t\\ x' \in \R}} \int_{\R^3} &  \big\| u(r,h)-u(r,0)\|_{2p}\|v^{(r,z)}(t,x)\|_{2p}\\
	& \times  \|u(r,z + h)\|_{2p} \big\| v^{(r,h)}(t,x'-z)-v^{(r,0)}(t,x' - z) \big\|_{2p} |h|^{2H - 2} d h dz dx,\\
	K_3\coloneqq \sup_{\substack{0\leq r \leq t\\ x' \in \R}} \int_{\R^3} & \|u(r,z + h)\|_{2p} \big\| v^{(r,h)}(t,x-z)-v^{(r,0)}(t,x - z) \big\|_{2p}\\
	& \times    \big\| u(r,h)-u(r,0)\|_{2p}\|v^{(r,z)}(t,x')\|_{2p} |h|^{2H - 2} d h dz dx,
	\shortintertext{and}
	K_4\coloneqq \sup_{\substack{0\leq r \leq t\\ x' \in \R}} \int_{\R^3} & \|u(r,z + h)\|_{2p}^2 \big\| v^{(r,h)}(t,x-z)-v^{(r,0)}(t,x - z) \big\|_{2p}\\
	& \times  \big\| v^{(r,h)}(t,x' - z)-v^{(r,0)}(t,x' - z) \big\|_{2p} |h|^{2H - 2} dh dz dx.
\end{align*}
It follows from Lemmas \ref{Lemma-difference of u}
and \ref{Lemma-v} that
\begin{align*}
	K_1 \leq & \sup_{0\leq r \leq t}\int_{\R}   \big\| u(r,h)-u(r,0)\|_{2p}^2 |h|^{2H - 2} dh \times \sup_{\substack{0\leq r\leq t\\ z \in \R}}\int_{\R} \|v^{(r,z)}(t,x)\|_{2p} dx \\
	&\times \sup_{\substack{0\leq r\leq t\\ x' \in \R}} \int_{\R} \|v^{(r,z)}(t,x')\|_{2p} dz \leq C.
\end{align*}

\noindent Due to Cauchy-Schwarz's inequality and Lemmas \ref{Lemma-u}, \ref{Lemma-difference of u}, \ref{Lemma-v}, and \ref{Lemma-difference of v}\ref{lmm-diff-va}), we have
\begin{align*}
K_2 \leq  & \sup_{\substack{0\leq r\leq t\\ z' \in \R}}\|u(r,z')\|_{2p} \times \sup_{\substack{0\leq r\leq t\\ z \in \R}} \int_{\R}\|v^{(r,z)}(t,x)\|_{2p}dx\\
 & \times \sup_{0\leq r\leq t} \bigg(\int_{\R} \big\| u(r,h)-u(r,0)\|_{2p}^2 |h|^{2H - 2} dh\bigg)^{1/2} \\
	& \times \sup_{0\leq r\leq t}  \bigg(\int_{\R}\Big( \int_{\R} \big\| v^{(r,h)}(t,z'')-v^{(r,0)}(t,z'') dz'' \Big)^2  \big\|_{2p} |h|^{2H - 2} d h  \bigg)^{1/2} \leq C,
\end{align*}
and, with a change of variable $x - z \to x''$,
\begin{align*}
	K_3 \leq  & \sup_{\substack{0\leq r\leq t\\ z' \in \R}}\|u(r,z')\|_{2p} \times \sup_{\substack{0\leq r\leq t\\ x' \in \R}} \int_{\R}\|v^{(r,z)}(t,x')\|_{2p}dz  \\
	&\times \sup_{0\leq r \leq t}\bigg(\int_{\R} \big\| u(r,h)-u(r,0)\|_{2p}^2 |h|^{2H - 2} dh\bigg)^{1/2} \\
	& \times \sup_{0\leq r \leq t} \bigg(\int_{\R}\Big( \int_{\R} \big\| v^{(r,h)}(t,x'')-v^{(r,0)}(t,x'') d x'' \Big)^2  \big\|_{2p} |h|^{2H - 2} d h  \bigg)^{1/2} \leq C.
\end{align*}

\noindent Finally, preforming a change of variables $(x - z, x' - z) \to (x'' , z'')$, and using Cauchy-Schwarz's inequality and Lemmas \ref{Lemma-u} and \ref{Lemma-difference of v}\ref{lmm-diff-va}),
\begin{align*}
	K_4\leq & \sup_{\substack{0\leq r\leq t\\ z' \in \R}}\|u(r,z')\|_{2p}^2 \times \sup_{0\leq r \leq t} \bigg(\int_{\R} \Big( \big\| v^{(r,h)}(t,z'') - v^{(r,0)}(t,z'') \big\|_{2p}d z'' \Big)^2 |h|^{2H - 2} dh \bigg)^{1/2} \\
	& \times \sup_{0\leq r \leq t} \bigg(\int_{\R}\Big(\int_{\R} \big\| v^{(r,h)}(t,x'')-v^{(r,0)}(t,x'') \big\|_{2p} dx'' \Big)^2 |h|^{2H - 2} dh \bigg)^{1/2} \leq C.
\end{align*}
The proof of Proposition \ref{prop_incre}\ref{prof_incre-d}) is complete.

\medskip

c) Without loss of generality, assume that $r \leq s$.
Then, as a result of Lemmas \ref{lmm-dprod}\ref{D2-prod}), \ref{trans-u}, \ref{lmm_trans-v}\ref{trans-v1'}) and \ref{lmm_trans-v}\ref{trans-v2}), and H\"{o}lder's inequality, we deduce that
\begin{align}\label{square_1-4}
\big\| \square_{h,\hbar}(r, z, s, y, t, x) \big\|_p \leq I_1 + I_2 + I_3 + I_4
\end{align}
where
\begin{gather*}
I_1 \coloneqq  \| u(r, h) - u(r,0)\|_{3p} \big\|v^{(r,0)}(s,y + \hbar - z) - v^{(r,0)}(s,y - z)\big\|_{3p} \|v^{(s,y)}(t,x)\|_{3p},\\
\begin{aligned}
I_2 \coloneqq \|u(r,z + h)\|_{3p} \big\| & v^{(r,h)}(s,y + \hbar - z) - v^{(r,h)}(s,y - z) \\
& - v^{(r,0)}(s,y + \hbar - z) + v^{(r,0)}(s,y - z)\big\|_{3p} \|v^{(s,y)}(t,x)\|_{3p},
\end{aligned}\\
I_3 \coloneqq   \|u(r, h) - u(r,0)\|_{3p} \big\| v^{(r,z)}(s,y + \hbar)\|_{3p}  \big\| v^{(s,\hbar)}(t,x - y) - v^{(s, 0)}(t,x - y)\big\|_{3p},
\shortintertext{and}
\begin{aligned}
I_4 \coloneqq & \|u(r,z + h)\|_{3p} \big \|v^{(r,z)}(s,y + \hbar) - v^{(r,z + h)}(s, y + \hbar)\big\|_{3p}\\
& \times \big\| v^{(s,\hbar)}(t,x - y) - v^{(s,0)}(t,x - y)\big\|_{3p}.
\end{aligned}
\end{gather*}
Then, preforming a change of variable $(z - y, y) \to (\widetilde{z}, y)$, and applying Lemmas \ref{Lemma-difference of u}, \ref{Lemma-v} and \ref{Lemma-difference of v}\ref{lmm-diff-vb}), we get
\begin{align*}
	& \int_{\R^2} \sup_{y \in \R}\Big(\int_{\R^2}  I_1 dx dz \Big)^2 |\hbar|^{2H - 2} |h|^{2H - 2} dh d \hbar\\
	= &\int_{\R} \Big(\int_{\R}  \big\|v^{(r,0)}(s, \hbar - \widetilde{z}) - v^{(r,0)}(s,\widetilde{z})\big\|_{3p} d \widetilde{z}\Big)^2   \sup_{y \in \R}\Big(\int_{\R}\|v^{(s,y)} (t,x)\|_{3p} dx  \Big)^2 |\hbar|^{2H - 2} d \hbar\\
	 & \times  \int_{\R} \| u(r, h) - u(r,0)\|_{3p}^2 |h|^{2H - 2} dh  \leq C,
\end{align*}
with $C > 0$ depending on $(t,p,H)$. Similarly, one can show that
\begin{align*}
	\int_{\R^2} \sup_{y \in \R}\Big(\int_{\R^2}  I_k dx dz \Big)^2 |\hbar|^{2H - 2} |h|^{2H - 2} dh d \hbar < C \quad \text{for all } k = 2,3,4.
\end{align*}
This completes that proof of Proposition \ref{prop_incre}\ref{prop_incre-b}).

\medskip

d) Similarly as in the proof of Proposition \ref{prop_incre}\ref{prop_incre-b}), we decompose $\|\square_{h,\hbar} (r,z,s,y,t,x)\|_p$ by \eqref{square_1-4}. Then, one can deduce that
\begin{align*}
	\int_{\R^2} \Big( & \int_{\R^2}  I_1 dy dz \Big)^2  |\hbar|^{2H - 2} |h|^{2H - 2} dh d \hbar = \int_{\R} \| u(r, h) - u(r,0)\|_{3p}^2|h|^{2H - 2} dh\\
	&\times \int_{\R}\Big(\int_{\R}   \big\|v^{(r,0)}(s,\widetilde{z} + \hbar) - v^{(r,0)}(s,\widetilde{z} ) \big\|_{3p} d\widetilde{z}\Big)^2 \Big(\int_{\R} \|v^{(s,y)}(t,x)\|_{3p} dy \Big)^2 |\hbar|^{2H - 2}  d \hbar < C.
\end{align*}
Again, similar arguments can be applied to the integrations of $I_2$, $I_3$ and $I_4$, and we skip them for conciseness. The proof of this proposition is complete.
\end{proof}

\section{Proofs of the main results}
\label{section-main}

In this section, we present the proof of Theorems \ref{EGD}, \ref{limit-cov}, \ref{QCLT} and \ref{FCLT}.

\subsection{Spatial Ergodicity---Proof of Theorem \ref{EGD}}
\label{prf_egd}

In this section, we include the proof of Theorem \ref{EGD}.
Recall that the stationarity of $\{u(t,x)\}_{x \in \bR}$ was proved in
Lemma \ref{trans-u}. In this section, we prove that this process is also ergodic. For this, we use a version of the ergodicity criterion given by \cite[Lemma 7.2]{CKN+09}, as stated in \cite[Lemma 4.2]{BZ23}. More precisely, we prove that:
\[
\lim_{R \to \infty}\frac{1}{R^2} {\rm Var}(U_R) =0,
\quad
\mbox{where} \quad
U_R=\int_{-R}^{R}g\Big(\sum_{j=1}^{k}b_j u(t,x+\zeta_j)\Big)dx,
\]
where $k$ is an arbitrary positive integer, $b_1,\ldots,b_k \in \bR$ and $\zeta_1,\ldots,\zeta_k \in \bR$ are arbitrary and $g \in \{\cos,\sin\}$.

Without loss of generality, we assume that $g(x)=\cos x$, the case when $g(x)=\sin x$ being similar. By the Gaussian-Poincar\'e inequality (see e.g, \cite[Exerice 2.11.1]{NP12}),
\[
{\rm Var}(U_R) \leq \bE\|DU_R \|_{\cH}^2=\int_0^t \int_{\bR^2}\bE|D_{r,y}U_R-D_{r,z}U_R|^2 |y-z|^{2H-2} dydz dr.
\]
Using the chain rule $D \varphi(F)=\varphi'(F)DF$, we see that
\[
D_{r,y}U_R=\int_{-R}^R \sin \Big( \sum_{j=1}^{k} b_j u(s,x+\zeta_j) \Big ) \sum_{j=1}^{k} b_j D_{r,y} u(t,x+\zeta_j)dx.
\]
Using Minkowski inequality, Cauchy-Schwarz inequality and the fact that $|\sin(x)|\leq 1$,
\begin{align*}
& \bE|D_{r,y}U_R-D_{r,z}U_R|^2 =\|D_{r,y}U_R-D_{r,z}U_R\|_2^2 \\
& \quad \leq
\left(\int_{-R}^{R} \left\|\sin\left(\sum_{j=1}^{k}b_j u(t,x+\zeta_j) \right) \sum_{j=1}^{k}b_j \left(D_{r,y}u(t,x+\zeta_j)-
D_{r,z}u(t,x+\zeta_j) \right) \right\|_2 dx \right)^2 \\
& \quad \leq
\left(\int_{-R}^{R} \sum_{j=1}^{k}|b_j| \big\|D_{r,y}u(t,x+\zeta_j)-
D_{r,z}u(t,x+\zeta_j)\big\|_4  dx \right)^2,
\end{align*}
and therefore,
\[
{\rm Var}(U_R) \leq C \sum_{j = 1}^k b_j^2 \times \sup_{\zeta\in \R} \int_0^t \int_{\bR^2}
\left(\int_{-R}^{R}  \big\| \Delta_h(r,z,t,x + \zeta) \big\|_4  dx \right)^2 |h|^{2H-2} dh dz dr,
\]
where $\Delta$ is defined as in \eqref{def_delta}. Notice that Proposition \ref{prop_incre}\ref{prof_incre-d}) yields that
\begin{align*}
& \sup_{\zeta\in \R} \int_{\bR^2} \left(\int_{-R}^{R}  \big\| \Delta_h(r,z,t,x + \zeta) \big\|_4  dx \right)^2 |h|^{2H-2} dh dz \\
&= \sup_{\zeta\in \R} \int_{\bR^2}\left( \int_{-R}^{R} \int_{-R}^{R}
\big\|\Delta_h(r,z,t,x_1 + \zeta)\big\|_4 \big\|\Delta_h(r,z,t,x_2 + \zeta)\big\|_4 dx_1 dx_2\right) |h|^{2H-2}dhdz \\
&\leq \sup_{\zeta\in \R} \int_{-R}^{R} \left(\int_{\bR^3}
\big\|\Delta_h(r,z,t,x_1)\big\|_4 \big\|\Delta_h(r,z,t,x_2 + \zeta)\big\|_4 |h|^{2H-2}dhdzdx_1 \right)dx_2  \\
& \leq  \int_{-R}^R \left(\sup_{y \in \R }  \int_{\bR^3} \big\| \Delta_h(r,z,t, x_1) \big\|_4  \big\| \Delta_h(r,z,t, y) \big\|_4 |h|^{2H-2} dh dz dx_1 \right)  d x_2 \leq C R,
\end{align*}
and thus
\begin{align*}
	\mathrm{Var} (U_R) \leq C R.
\end{align*}
It follows that
\begin{align*}
	\frac{1}{R^2} \mathrm{Var} (U_R) \leq C R^{-1} \to 0, \quad \text{as }R\to\infty.
\end{align*}
The proof of Theorem \ref{EGD} is complete.

\subsection{Limiting covariance---Proof of Theorem \ref{limit-cov}}
\label{sec_limit_cov}

In this section, we give the proof of Theorem \ref{limit-cov}.
Note that
\[
\sigma_R^2(t)={\rm Var}(F_R(t))=\int_{-R}^{R} \int_{-R}^R \rho_t(x-y)dxdy
\]
where
\begin{equation}
\label{def-rho}
\rho_{t}(x-y) \coloneqq \bE \left[ \big(u(t,x)-1\big) \big(u(t,y)-1\big) \right]
	= \sum_{n\geq 1}\frac{1}{n!}\gamma_n(t,x-y)
\end{equation}
and
\begin{align}
\label{eq-inner}
\nonumber
& \gamma_n(t,x-y) \coloneqq (n!)^2 \langle \widetilde f_n(\cdot,t,x), \widetilde f_n(\cdot,t,y) \rangle_{\cH^{\otimes n}}\\
\nonumber
& \quad = n!C_H^n \int_{T_n(t)}\int_{\bR^n}
	\cF f_n(\cdot,t,x)(\xi_1,\ldots,\xi_n) \overline{\cF f_n(\cdot,t,y)(\xi_{1},\ldots,\xi_{n})}
	\prod_{j=1}^{n}|\xi_j|^{1-2H} d\pmb{\xi}_n d\pmb{t}_n\\
\nonumber
& \quad = n! C_H^n \int_{T_n(t)}\int_{\bR^n}
	e^{-i(\xi_1+\ldots+\xi_n) (x-y)} \prod_{j=1}^n |\cF G_{t_{j+1}-t_j}(\xi_1+\ldots+\xi_j)|^2
\prod_{j=1}^{n}|\xi_j|^{1-2H} d\pmb{\xi}_n d\pmb{t}_n \\
& \quad =n! C_H^n \int_{T_n(t)}\int_{\bR^n}
	e^{-i\eta_n (x-y)} \prod_{j=1}^n \frac{\sin^2((t_{j+1}-t_j)|\eta_j|)}{|\eta_j|^2}
\prod_{j=1}^{n}|\eta_j-\eta_{j-1}|^{1-2H} d\pmb{\eta}_n d\pmb{t}_n,
\end{align}
with convention $t_{n+1}=t$ and $\eta_0=0$.
This shows that
$\alpha_n(t,x-y)$ and $\rho_{t}(x-y)$ depend on $x$ and $y$ only through the difference $x-y$.
In particular, $\{u(t,x);x \in \bR\}$ is a (wide-sense) stationary process with covariance function $\rho_t$.

\begin{proof}[Proof of Theorem \ref{limit-cov}]
The proof is divided in three steps.

{\bf Step 1.}
In this step, we prove \eqref{lim-cov} in the case $t=s$. We write
\begin{equation}
\label{eq-sigma}
\sigma_{R}^2(t)=\sum_{n\geq 1}\frac{1}{n!}\int_{-R}^{R}\int_{-R}^R \gamma_n(t,x-y)dxdy.
\end{equation}
Note that
\begin{equation}
\label{exp}
\int_{-R}^R \int_{-R}^R e^{-i \xi(x-y)}dxdy=\left|\int_{-R}^{R}e^{-i\xi x}dx\right|^2=\frac{4\sin^2(R|\xi|)}{|\xi|^2}=4\pi R\ell_R(\xi)
\end{equation}
where $\ell_R(x) \coloneqq (\pi R|x|^2)^{-1}\sin^2(|x|R)$ is an approximation of the identity as $R\to \infty$; see \cite[Lemma 2.1]{NZ20-1}. On the other hand, \cite[Theorem 3.2]{SSX20} shows that for every $n\geq 1$,
\begin{align*}
\int_{T_n(t)}\int_{\bR^n}
	\bigg| e^{-i\eta_n (x-y)} \prod_{j=1}^n \frac{\sin^2((t_{j+1}-t_j)|\eta_j|)}{|\eta_j|^2}
\prod_{j=1}^{n}|\eta_j-\eta_{j-1}|^{1-2H}\bigg| d\pmb{\eta}_n d\pmb{t}_n < \infty.
\end{align*}
Thus, by Fubini theorem,
\begin{align}\label{gmm>0}
& \frac{1}{n!}\int_{-R}^R \int_{-R}^R \gamma_n(t,x-y)dxdy\nonumber\\
= &  C_H^n \int_{T_n(t)} \int_{\bR^n}
\left(\int_{-R}^R \int_{-R}^R
e^{-i\eta_n(x-y)}dxdy\right)
\prod_{j=1}^{n}
\frac{\sin^2((t_{j+1}-t_j)|\eta_j|)}{|\eta_j|^2}
\prod_{j=1}^n |\eta_j-\eta_{j-1}|^{1-2H}d \pmb{\eta}_n d\pmb{t}_n \nonumber\\
 = & 4C_H^n \int_{T_n(t)} \int_{\bR^n}
\frac{\sin^2 (R|\eta_n|)}{|\eta_n|^2}
\prod_{j=1}^{n}
\frac{\sin^2((t_{j+1}-t_j)|\eta_j|)}{|\eta_j|^2}
\prod_{j=1}^n |\eta_j-\eta_{j-1}|^{1-2H}d \pmb{\eta}_n d\pmb{t}_n.
\end{align}

We treat separately the case  $n = 1$.
From \cite[Page 30]{BY23-2}, we know that for any $t_1 \in [0,t]$, $\theta \in (H,\frac{1}{2})$ and $\varepsilon \in (0,\frac{\pi}{4t})$,
\[
\int_{\bR} \frac{\sin^2(R|\eta_1|)}{|\eta_1|^2}
\frac{\sin^2((t-t_1)|\eta_1|)}{|\eta_1|^2} |\eta_1|^{1-2H}d\eta_1
\leq C_{\varepsilon,\theta,H} (t-t_1)^2 R^{2\theta} +C_{\varepsilon,H}.
\]
Therefore,
\[
\lim_{R \to \infty}\frac{1}{R}\int_{-R}^{R} \int_{-R}^{R}\gamma_1(t,x-y)dxdy \leq 4C_H \lim_{R \to \infty} R^{2\theta - 1}\int_0^t \big( C_{\varepsilon,\theta,H} (t-t_1)^2 +C_{\varepsilon,H}\big) dt_1 =0.
\]
Next, we examine the terms corresponding to $n\geq 2$.
For any $n\geq 2$, denote
\[
g_{\pmb{t}_n}^{(n)}(\eta) \coloneqq \frac{\sin^2((t-t_n)|\eta|)}{|\eta|^2}
\int_{\bR^{n-1}} \prod_{j=1}^{n-1}\frac{\sin^2((t_{j+1}-t_j)|\eta_j|)}{|\eta_j|^2} \prod_{j=1}^{n}|\eta_j-\eta_{j-1}|^{1-2H} d\pmb{\eta}_{n-1},
\]
with
\begin{align*}
g_{\pmb{t}_n}^{(n)}(0) \coloneqq & \lim_{\eta \to 0} g_{\pmb{t}_n}^{(n)}(\eta)\\
 = & (t-t_n)^2
\int_{\bR^{n-1}} \prod_{j=1}^{n-1}\frac{\sin^2((t_{j+1}-t_j)|\eta_j|)}{|\eta_j|^2} \prod_{j=1}^{n-1}|\eta_j-\eta_{j-1}|^{1-2H} |\eta_{n-1}|^{1-2H}d\pmb{\eta}_{n-1},
\end{align*}
where the last equality is due to the fact that $\lim_{x \to 0}\sin x/x=1$.
Then, $g_{\cdot}^{(n)}(*)$ is a non-negative function on $T_n(t)\times \bR$, and thus
\begin{align}
\label{alpha_1}
\frac{1}{n!}\int_{-R}^{R} & \int_{-R}^R \gamma_n(t,x-y)dxdy\nonumber \\
= &4\pi C_H^n R \int_{T_n(t)}\int_{\bR^n} \ell_R(\eta_n) \prod_{j=1}^{n}\frac{\sin^2((t_{j+1}-t_j)|\eta_j|)}{|\eta_j|^2} \prod_{j=1}^{n}|\eta_j-\eta_{j-1}|^{1-2H} d\pmb{\eta}_n d\pmb{t}_n \nonumber\\
= &4\pi C_H^n R \int_{T_n(t)} \int_{\bR}\ell_R(\eta_n) g_{\pmb{t}_n}^{(n)}(\eta_n)d\eta_n d\pmb{t}_n =
4\pi C_H^n R \int_{T_n(t)} \big(\ell_R*g_{\pmb{t}_n}^{(n)}\big)(0) d\pmb{t}_n.
\end{align}

\noindent Using \eqref{prod}, we obtain that
\begin{align}\label{ine_gtn}
g_{\pmb{t}_n}^{(n)}(\eta_n)&\leq\sum_{\pmb{\alpha}_n \in D_n}
\frac{\sin^2((t-t_n)|\eta_n|)}{|\eta_n|^2}|\eta_n|^{\alpha_n}\prod_{j=1}^{n-1} \left( \int_{\bR} \frac{\sin^2((t_{j+1}-t_j)|\eta_j|)}{|\eta_j|^2} |\eta_j|^{\alpha_j} d\eta_j\right)\nonumber\\
&\leq C^{n-1} \sum_{\pmb{\alpha}_n\in D_n}
\frac{\sin^2((t-t_n)|\eta_n|)}{|\eta_n|^2}|\eta_n|^{\alpha_n}  \prod_{j=1}^{n-1}(t_{j+1}-t_j)^{1-\alpha_j} ,
\end{align}
where for the last inequality we used \eqref{sin-int}. It is not difficult to see that $g_{\pmb{t}_n}^{(n)}$ is continuous, locally bounded and integrable on $\bR$.
Recall that $\ell_R$ is an approximation of the identity as $R\to \infty$.
Hence, for any $n\geq 2$,
\[
\lim_{R\to \infty}\big(\ell_R*g_{\pmb{t}_n}^{(n)}\big)(0) = g_{\pmb{t}_n}^{(n)}(0).
\]

\noindent  Combining \eqref{eq-sigma} and \eqref{alpha_1}, by using the dominated convergence theorem, we see that
\begin{align}\label{limit_r}
\lim_{R\to\infty}\frac{\sigma_R^2(t)}{R}=
\lim_{R\to\infty} 4\pi\sum_{n\geq 2}  C_H^n \int_{T_n(t)} \big(\ell_R*g_{\pmb{t}_n}^{(n)}\big)(0) d\pmb{t}_n = 4\pi\sum_{n\geq 2}C_H^n\int_{T_n(t)} g_{\pmb{t}_n}^{(n)}(0) d\pmb{t}_n.
\end{align}

As the dominated convergence theorem is applied in \eqref{limit_r}, one needs to justify the applicability. In other words, we have to find a sequence of functions $\{h_n\}_{n\geq 2}$ on $T_n(t)$ such that $\big(\ell_R*g_{\pmb{t}_n}^{(n)}\big)(0)\leq h_n(\pmb{t}_n)$ for all $\pmb{t}_n \in T_n(t)$ and $n\geq 1$, and
\begin{equation}
\label{series-h}
\sum_{n\geq 2}C_H^n\int_{T_n(t)} h_n(\pmb{t}_n)d\pmb{t}_n <\infty.
\end{equation}
In particular, this shows that
\[
K(t) = 4\pi\sum_{n\geq 2}C_H^n\int_{T_n(t)} g_{\pmb{t}_n}^{(n)}(0) d\pmb{t}_n < \infty.
\]

\noindent Thanks to \eqref{ine_gtn} and the inequality $\frac{\sin^2((t-t_n)|\eta|)}{|\eta|^2}\leq (t-t_n)^2$, we have:
\begin{align*}
\big(\ell_R*g_{\pmb{t}_n}^{(n)}\big)(0)& \leq C^{n-1}\sum_{\pmb{\alpha}_n\in D_n}\prod_{j=1}^{n-1}(t_{j+1}-t_j)^{\alpha_j}
\int_{\bR}\frac{\sin^2(R|\eta_n|)}{\pi R|\eta|^2}\frac{\sin^2((t-t_n)|\eta_n|)}{|\eta_n|^2}|\eta_n|^{\alpha_n}d\eta_n \\
&\leq  C^{n-1}\sum_{\pmb{\alpha}_n \in D_n}\prod_{j=1}^{n-1}(t_{j+1}-t_j)^{\alpha_j}(t-t_n)^2
\int_{\bR}\frac{\sin^2(R|\eta_n|)}{\pi R|\eta_n|^2}|\eta_n|^{\alpha_n}d\eta_n \\
&\leq  C^{n}\sum_{\pmb{\alpha}_n \in D_n} \prod_{j=1}^{n-1}(t_{j+1}-t_j)^{1-\alpha_j}(t-t_n)^2
 R^{-\alpha_n}\\
& \leq C^n\sum_{\pmb{\alpha}_n \in D_n}\prod_{j=1}^{n-1}(t_{j+1}-t_j)^{1-\alpha_j}
(t-t_n)^2 \eqqcolon h_n(\pmb{t}_n),
\end{align*}
for any $R\geq 1$, where identity \eqref{sin-int} is used for the second last line.
As a consequence of identity \eqref{beta-int} and the fact that $\Gamma(an+b+1) \geq C^n (n!)^a$ for any $a>0,b\in \bR$, we get:
\begin{align*}
\int_{T_n(t)}h_n(\pmb{t}_n) d\pmb{t}_n \leq \frac{C^n}{(n!)^{2H+1}}(t^{(1+2H)n+1}+t^{(1+2H)n+2-2H}),
\end{align*}
which yields \eqref{series-h}.

\medskip

{\bf Step 2.} In this step, we show that
\[
\lim_{R\to \infty}\frac{\sigma_R^2(t)}{R}>0.
\]
Recall \eqref{gmm>0}, we have
\[
\int_{-R}^R \int_{-R}^R\gamma_n(t,x-y)dxdy > 0,
\]
for all $n \geq 2$, $R > 0$ and $ t > 0$.
Then, taking account of \eqref{alpha_1}, we can write
\begin{align*}
\lim_{R\to \infty}\frac{\sigma_R^2(t)}{R} > & \lim_{R\to \infty}\frac{1}{2R}\int_{-R}^R \int_{-R}^R \gamma_2(t,x-y)dxdy
=  4\pi C_H^2 \int_{0<r<s<t} g^{(2)}_{r,s}(0)drds\\
= & 4\pi C_H^2\int_{0<r<s<t}
\int_{\bR}\frac{\sin^2((s-r)\eta_1)}{|\eta_1|^{2}}|\eta_1|^{2(1-2H)}d\eta_1 drds\\
= & C \int_{0<r<s<t} (s-r)^{4H-1}drds>0 ,
\end{align*}
 where $C>0$ is a constant depending on $H$, and the last line is due to \eqref{sin-int} and the fact that $H>1/4$.

\medskip

{\bf Step 3.} In this step, we complete the proof of \eqref{lim-cov}. Without loss of generality, assume that $s \leq t$. Similarly to \eqref{eq-inner},
\[
\bE[F_R(t) F_R(s)]=\sum_{n\geq 1}\frac{1}{n!} \int_{-R}^R \int_{-R}^R \widehat{\gamma}_{n}(t,s,x-y)dxdy,
\]
where
\begin{align*}
\widehat{\gamma}_n(t,s,x-y)&=(n!)^2 \langle \widetilde{f}_n(\cdot,t,x),
\widetilde{f}_n(\cdot,s,y)\rangle_{\cH^{\otimes n}}\\
&=n! C_H^n \int_{T_n(s)} \int_{\bR^{n}} e^{-i \eta_n (x-y)} \prod_{j=1}^{n-1}\frac{\sin^2((t_{j+1}-t_j)|\eta_j|)}{|\eta_j|^2} \\
& \quad \quad \quad \quad \quad \frac{\sin((t-t_n)|\eta_n|) \sin((s-t_n)|\eta_n|)}{|\eta_n|^2} \prod_{j=1}^{n}
|\eta_j-\eta_{j-1}|^{1-2H}d\pmb{\eta}_n d\pmb{t}_n.
\end{align*}
The same argument as for \eqref{alpha_1} shows that
\begin{align*}
\frac{1}{n!}\int_{-R}^{R} \int_{-R}^R \widehat{\gamma}_n(t,s,x-y)dxdy=
4\pi C_H^n R \int_{T_n(s)} \big(\ell_R*\widehat{g}_{\pmb{t}_n}^{(n)}\big)(0) d\pmb{t}_n,
\end{align*}
where
\begin{align*}
\widehat{g}_{\pmb{t}_n}^{(n)}(\eta) \coloneqq & \frac{\sin((t-t_n)|\eta|)\sin((s-t_n)|\eta|)}{|\eta|^2}
\\
& \times \int_{\bR^{n-1}} \prod_{j=1}^{n-1}\frac{\sin^2((t_{j+1}-t_j)|\eta_j|)}{|\eta_j|^2}  \prod_{j=1}^{n}|\eta_j-\eta_{j-1}|^{1-2H} d\pmb{\eta}_{n-1}.
\end{align*}
Then, relation \eqref{lim-cov} follows by the dominated convergence theorem, as in Step 1. The proof of this theorem is complete.
\end{proof}

\begin{remark}
There might be an alternative method for proving Theorem \ref{limit-cov}, which would give a different representation of the limiting covariance. We explain this method here. If
\begin{equation}
\label{rho-int}
\int_{\bR}|\rho_t(x)|dx<\infty,
\end{equation}
then by the dominated convergence theorem
\[
\frac{\sigma_{R}^2(t)}{R}=\int_{\bR}\rho_t(x) \frac{|B_R \cap B_R(-x)| }{R}dx \to 2 \int_{\bR}\rho_t(x)dx \quad \mbox{as} \quad R \to \infty
\]
where $B_R=[-R,R]$ (see \cite[Page 27]{NZ20-1}). Recalling definition \eqref{def-rho} of $\rho_t(x)$, \eqref{rho-int} follows if one can show that
\[
\sum_{n\geq 1}\frac{1}{n!} \int_{\bR}|\gamma_n(t,x)|dx<\infty.
\]
Unfortunately, we could not prove that $\gamma_n(t,\cdot)$ is integrable on $\bR$.
\end{remark}

\begin{remark}
Suppose we can exchange the order of integrals arbitrarily and treat the Diract distribution $\delta$ as a regular function. Then, we can write
\begin{align*}
4\pi C_H^n\int_{T_n(t)}g_{\pmb{t}_n}^{(n)}(0)d\pmb{t}_n =&4\pi C_H^n\int_{T_n(t)}\int_{\bR}\delta_0(\eta_n)g_{\pmb{t}_n}^{(n)}(\eta_n)d\eta_nd\pmb{t}_n\\
=&2C_H^n\int_{T_n(t)}\int_{\bR}\Big(\int_{\bR}e^{-i\eta_n x}\mathbf{1}_{\bR}(x)dx\Big) g_{\pmb{t}_n}^{(n)}(\eta_n)d\eta_nd\pmb{t}_n\\
=&\frac{2}{n!}\int_{\bR}\langle f_n(\cdot, t,x), f_n(\cdot, t,0)\rangle_{\cH}dx.
\end{align*}
Thus it is natural to conjecture that
\begin{align}\label{cj_limit}
\lim_{R\to \infty}\frac{1}{R}\sigma_R^2=2\int_{\bR}\mathrm{Cov}(u(t,x),u(t,0))dx.
\end{align}
Actually, this has been confirmed in parabolic cases (see \cite{NSZ20,NZ20-1}) and also in hyperbolic cases assuming the non-negative correlation in space of driven noises (see \cite{BNQZ22}). The aforementioned results rely on the nonnegativity of $\gamma_n(t,x)$, which does not hold in our setting. This prevents us to provide a proof for \eqref{cj_limit}. We expect it can be verified in the future.
\end{remark}

\subsection{Quantitative CLT---Proof of Theorem \ref{QCLT}}
\label{proof_qclt}

In this section, we prove Theorem \ref{QCLT}.
Applying \cite[Proposition 2.4]{NXZ22},
we get:
\[
d_{TV}\left(\frac{F_R(t)}{\sigma_R(t)},Z \right) \leq \frac{2\sqrt{3}}{\sigma_R^2(t)} \sqrt{\cA},
\]
where
\[
\cA=C_H^3 \int_{[0,t]^3}\cA_0^*(r,s,\theta)drdsd\theta,
\]
and
\begin{align*}
 \cA_0^*(r,s,&\theta) = \int_{\bR^6} \|D_{r,z}F_R(t)- D_{r,z'}F_R(t)\|_4  \|D_{\theta,w}F_R(t)-D_{\theta,w'}F_R(t)\|_4 \\
& \times \|
D_{(r,z) ,(s,y)}^2  F_R(t)-
D_{(r,z) ,(s,y')}^2 F_R(t)-
D_{(r,z'),(s,y)}^2  F_R(t)+
D_{(r,z'),(s,y')}^2 F_R(t) \|_4 \\
&  \times \|
D_{(\theta,w) ,(s,y)}^2  F_R(t)-
D_{(\theta,w) ,(s,y')}^2 F_R(t)-
D_{(\theta,w'),(s,y)}^2  F_R(t)+
D_{(\theta,w'),(s,y')}^2 F_R(t) \|_4  \\
&  \times |y-y'|^{2H-2}|z-z'|^{2H-2}|w-w'|^{2H-2}dydy' dzdz' dw dw'.
\end{align*}
Since Theorem \ref{limit-cov} concludes that $\sigma_R^2(t) \sim C R$ as $R \to \infty$, it is enough to show that $\cA \leq C R$.

By Minkowski's inequality,
we get $\cA_0^*(r,s,\theta) \leq \cA_0(r,s,\theta)$, where $\cA_0$ can be written, after a change of variables, as follows:
\begin{align*}
\cA_0&(r,s,\theta) =  \int_{[-R,R]^4} dx_1 dx_2 dx_3 dx_4  \int_{\bR^6} dydy' dzdz' dw dw' dw' |y'|^{2H-2}|z'|^{2H-2}|w'|^{2H-2}\\
 & \times \big\|\Delta_{z'}(r,z,t,x_1)\big\|_4  \big\|\Delta_{w'} (\theta,w,t,x_2) \big\|_4 \big\|
\square_{z',y'}(r,z,s,y,t,x_3)\big\|_4  \big\|
\square_{w',y'} (\theta, w, s, y, t, x_4) \big\|_4.
\end{align*}
where $\Delta$ and $\square$ are defined as in \eqref{def_delta} and \eqref{def_rectg}, respectively.

Hence,
\begin{align*}
\cA \leq \cA_1 + \cA_2 + \cA_3 + \cA_4,
\end{align*}
where
\begin{gather*}
\cA_1 \coloneqq \int_{0 <r\vee \theta < s < t} \cA_0 (r,s,\theta) drdsd \theta, \qquad \cA_2 \coloneqq \int_{0 < r < s < \theta < t} \cA_0 (r,s,\theta) dr ds d \theta,\\
\cA_3\coloneqq \int_{0<\theta<s<r <t} \cA_0 (r,s,\theta) drdsd \theta, \quad \text{and}\quad \cA_4 \coloneqq \int_{0 < s < r\wedge \theta < t} \cA_0 (r,s,\theta) drdsd \theta.
\end{gather*}

The estimates for $\cA_1,\dots, \cA_4$ are quite similar. Here we only provide a detailed deduction of the estimate for $\cA_1$.
\begin{align}\label{ineq-Cauchy}
\cA_0(r,s,\theta) \leq T_1 (r,\theta)^{\frac{1}{2}} T_2 (s,r)^{\frac{1}{2}} \int_{-R}^R T_3(s,\theta, x_4)^{\frac{1}{2}} dx_4
\end{align}
where
\begin{gather*}
\begin{aligned}
T_1(r, \theta)\coloneqq \int_{\bR^2} & \Big(  \sup_{z \in \bR} \int_{\bR} \big\| \Delta_{z'}(r,z,t,x_1) \big\|_4 dx_1 \Big)^2 \Big(\sup_{w \in \bR} \int_{\bR} \big\| \Delta_{w'} (\theta, w, t, x_2) \big\| dx_2 \Big)^2 \\
&\times |z'|^{2H-2}|w'|^{2H-2} dz'dw' < C,
\end{aligned}\\
T_2(s,r)\coloneqq \int_{\bR^2} \left( \sup_{y \in \bR} \int_{\bR^2} \big\| \square_{z',y'}(r, z, s, y, t,x_3) \big\|_4 dx_3 dz \right)^2 |y'|^{2H-2}|z'|^{2H-2} dy'dz' < C,
\shortintertext{and}
T_3(s,\theta, x_4) \coloneqq \int_{\bR^2} \left(  \int_{\bR^2} \big\|\square_{w',y'}(\theta, w, s, y, y, x_4)\big\|_4 dydw \right)^2 |y'|^{2H-2}|w'|^{2H-2} dy'dw' < C,
\end{gather*}
by using Proposition \ref{prop_incre}. This yields immediately that
 for any $0 < r \vee \theta < s < t$,
\begin{align*}
	\cA_0(r,s,\theta) \leq C \int_{-R}^R dx_4 = C R,
\end{align*}
and thus
\begin{align*}
	\cA_1 =\int_{\{0 < r < s\vee \theta < s\}}\cA_0(r,s,\theta)drdsd\theta \leq C R.
\end{align*}
Similar arguments ensure that $\cA_k\leq CR$ for $k = 2,3,4$. Therefore, $\cA \leq CR$. This completes the proof of Theorem \ref{QCLT}.

\subsection{Functional CLT---Proof of Theorem \ref{FCLT}} \label{sec_FCLT}

In this section, we include the proof of Theorem \ref{FCLT}. We fix $T>0$. It suffices to show the following properties:

\begin{enumerate}[(1)]
\item \label{tight}
The tightness of the collection of $C([0,T])$-valued random variables $\{\sigma_R^{-1} F_R(r); r\in [0, T]\}_{R \in \R_+}$.

\item \label{finit} The convergence in distribution of $\{(\sigma_R^{-1} F_R(t_1),\dots,\sigma_R^{-1} F_R(t_m)\}_{R\in \R_+}$ to $(\cG(t_1),\ldots, \cG(t_m))$ as $R \to \infty$, for all positive integer $m$ and for all $0 \leq t_1 \leq \dots \leq t_m \leq T$.
\end{enumerate}

\medskip

\begin{proof}[Proof of tightness]
By Kolmogorov-Chentsov criterion (\cite[Theorem 23.7]{Kal21}), it is enough to prove that
\begin{align*}
\|F_R(t)-F_R(s)\|_p \leq C R^{1/2} (t-s)^{1/2},
\end{align*}
for all $0 \leq s \leq t \leq T$, where $C>0$ is a constant that depends on $(T, p , H)$. Using \eqref{def-sol} and the convention $G_t(x)=0$ for $t<0$, we write:
\[
u(t,x)-u(s,x) = \int_0^t \int_{\bR}\big(G_{t-r}(x-y)-G_{s-r}(x-y)\big) u(r,y)W(dr,dy).
\]
By stochastic Fubini theorem,
\[
F_R(t)-F_R(s) = \int_0^t \int_{\bR} \left( \int_{-R}^{R} \big(G_{t-r}(x-y)-G_{s-r}(x-y)\big) dx\right) u(r,y)W(dr,dy).
\]
We use the following inequality: for any predictable process $S$,
\[
\left\|\int_0^{T} \int_{\bR}S(t,x)W(dt,dx)\right\|_p^2 \leq C_{p,H} \int_0^{T}\int_{\bR^2}\|S(t,x)-S(t,y)\|_p^2 |x-y|^{2H-2}dxdy dt.
\]
This inequality follows from the Burhholder-Davis-Gundy inequality for the stochastic integral with respect to $W$ (given by \cite[Theorem 2.9]{BJQ15}) followed by Minkowski inequality for the $\|\cdot\|_{p/2}$-norm. It follows that
\[
\|F_R(t)-F_R(s)\|_p^2 \leq C_{p,H} \int_0^t \int_{\bR^2}\|S(r,y)-S(r,z)\|_p^2 |y-z|^{2H-2}dydz dr,
\]
with
$S(r,y) =\left(\int_{-R}^R \big(G_{t-r}(x-y)-G_{s-r}(x-y)\big) dx\right) u(r,y)$.
It follows that:
\begin{align*}
	\|F_R(t)-F_R(s)\|_p^2 \leq C_{p,H} \big( I_1 + I_2 \big),
\end{align*}
where
\begin{align*}
	I_1 & \coloneqq \int_0^t \int_{\bR^2} \left(\int_{-R}^R \big(G_{t-r}(x-y) - G_{s-r}(x-y) - G_{t-r}(x-z) + G_{s-r}(x-z) \big) dx\right)^2\\
	& \quad \quad \quad \times  \|u(r,y)\|_p^2 |y-z|^{2H-2}dydz dr \\
	I_2 & \coloneqq \int_0^t \int_{\bR^2} \left(\int_{-R}^R \big(G_{t-r}(x-z) - G_{s-r}(x-z) \big) dx\right)^2  \|u(r,y) - u(r,z)\|_p^2 |y-z|^{2H-2}dydzdr.
\end{align*}

Note that
\begin{align}
\label{int_t-s}
\int_{-R}^R \big(G_{t-r}(x-z) - G_{s-r}(x-z) \big)  dx  =\frac{1}{2}\int_{-R}^{R}1_{\{s-r<|x-z|<t-r\}}dx \in [0,t-s].
\end{align}
Thus, due to Lemmas \ref{trans-u} and \ref{Lemma-difference of u}, and the fact that $G_{s-r}(x) \leq G_{t-r}(x)$, we deduce that
\begin{align*}
	I_2\leq & C (t - s) \int_{-R}^R \int_0^t \left(\int_{\R} G_{t-r} (x-z)dz \right) dr dx =  C (t - s) R  \int_0^t  (t-r)   d r \leq C (t - s) R.
\end{align*}

Next, we treat $I_1$. By \eqref{int_t-s},
\begin{align*}
\left|\int_{-R}^{R} \big( G_{t-r}(x-y) - G_{s-r}(x-y) - G_{t-r}(x-z) +G_{s-r}(x-z) \big) dx\right| \leq 2 (t-s).
\end{align*}
Using this bound and Lemma \ref{Lemma-u}, we find that
\begin{align*}
	I_{1} & \leq 2(t-s) \int_{-R}^R   \int_0^t \int_{\R^2} \big(|G_{t-r}(x-y) - G_{t-r}(x-z)| + |G_{s-r}(x-y) - G_{s-r}(x-z)| \big) \\
	& \qquad \qquad \qquad \qquad \qquad \times |y-z|^{2H-2} dydz dr dx
\end{align*}
Note that $|G_{t}(x) - G_{t}(y)| \in \{1/2,0\}$ for all $t \geq 0$ and $x,y\in \R$, thus $|G_t(x) - G_t(y)| = 2 |G_t(x) - G_t(y)|^2$. As a result,
\begin{align*}
I_{1} & \leq 4 (t-s) \bigg( \int_{-R}^R  \int_0^t \int_{\R^2}  |G_{t-r}(x-y) - G_{t-r}(x-z)|^2 |y-z|^{2H-2} dydz drdx \\
 & + \int_{-R}^R \int_0^t \int_{\R^2} |G_{s-r}(x-y) - G_{s-r}(x-z)|^2 |y-z|^{2H-2}dydzdr dx\bigg) \leq  C (t-s) R,
\end{align*}
where the last inequality is due to \cite[Inequality (58)]{BY23-2}.
\end{proof}

\begin{proof}[Proof of finite dimensional convergence]
The proof follows the same idea as in \cite[Section 4.2]{BNQZ22}. More precisely, it suffices to show that:
\begin{align*}
	\mathrm{Var} \big( \big\langle D F_R(t_i), - D L^{-1} F_R(t_j) \big\rangle_{\mathcal{H}} \big) \leq C R, \quad \mbox{for any $i,j=1,\ldots,m$}.
\end{align*}
This inequality is proved using the same argument as for $\cA$ in Section \ref{proof_qclt}, based on an estimate for ${\rm Var}(\langle DF, -DL^{-1}G\rangle_{\cH})$ for two random variables $F$ and $G$, which can be deduced similarly to the estimate derived for $F=G$ in the proof of \cite[Proposition 2.4]{NXZ22}.
\end{proof}

\end{document}